\documentclass[11pt,oneside]{article}
\usepackage{setspace,graphicx,amssymb,amsmath,latexsym,amsfonts,amscd,amsthm,multirow,ctable,mathdots,caption,array}
\usepackage{fancyhdr,tabularx,cite,mathrsfs}
\usepackage[headings]{fullpage}
\usepackage{stmaryrd}
\usepackage{rotating}
\usepackage{xcolor,colortbl}

\usepackage{hyperref}

\newcommand{\bea}{\begin{eqnarray}} 
\newcommand{\eea}{\end{eqnarray}} 
\newcommand{\bee}{\begin{eqnarray*}} 
\newcommand{\eee}{\end{eqnarray*}} 
\newcommand{\al}{\begin{align*}} 
\newcommand{\eal}{\end{align*}} 
\newcommand{\be}{\begin{equation}} 
\newcommand{\ee}{\end{equation}} 
 
\newcommand{\bem}{\begin{pmatrix}} 
\newcommand{\eem}{\end{pmatrix}} 

\def\a{\alpha} 
 
\def\c{\gamma}

\def\f{\phi}

\def\m{\mu}

\def\s{\sigma}            
\def\t{\tau} 
\def\th{\theta}

\def\D{\Delta}

\newcolumntype{R}{ >{$}r <{$}}
\newcolumntype{C}{ >{$}c <{$}}
\newcolumntype{L}{ >{$}l <{$}}
\newcolumntype{F}{>{\centering\arraybackslash}m{1.5cm}}

\def\ll{\ell}

\newcommand{\gt}[1]{\mathfrak{#1}}

\newcommand{\comment}[1]{}

\newcommand{\RR}{{\mathbb R}}
\newcommand{\CC}{{\mathbb C}}
\newcommand{\ZZ}{{\mathbb Z}}
\newcommand{\QQ}{{\mathbb Q}}
\newcommand{\HH}{{\mathbb H}}
\newcommand{\cA}{{\mathcal A}}

\newcommand{\Aut}{\operatorname{Aut}}

\newcommand{\Span}{\operatorname{Span}}

\def\rad{\operatorname{r}}

\newcommand{\tr}{\operatorname{{tr}}}
\newcommand{\str}{\operatorname{{str}}}

\newcommand{\Dih}{{\textsl{Dih}}}

\newcommand{\xmod}{{\rm \;mod\;}}
\newcommand{\new}{{\rm new}}

\newcommand{\PSL}{\operatorname{\textsl{PSL}}}    
\newcommand{\SL}{\operatorname{\textsl{SL}}}      
\newcommand{\Mpt}{\operatorname{\widetilde{\textsl{SL}_2}}}      
\newcommand{\Sp}{\operatorname{\textsl{Sp}}}      

\newcommand{\G}{\Gamma}	
\newcommand{\g}{\gamma}	

\newcommand{\rs}{{X}}	

\newcommand{\MM}{\mathbb{M}}	

\newcommand{\Th}{\Theta} 

\newtheorem{thm}{Theorem}[section]
\newtheorem{cor}[thm]{Corollary}
\newtheorem{lem}[thm]{Lemma}
\newtheorem{prop}[thm]{Proposition}
\newtheorem{conj}[thm]{Conjecture}

\theoremstyle{definition}

\theoremstyle{remark}

\numberwithin{equation}{section}

\begin{document}

\setstretch{1.4}

\title{
\vspace{-35pt}
    \textsc{
    		\huge{{W}eight {O}ne {J}acobi {F}orms and {U}mbral {M}oonshine}
        }
    }

\renewcommand{\thefootnote}{\fnsymbol{footnote}} 
\renewcommand{\thefootnote}{\arabic{footnote}} 



\author{
	Miranda C. N. Cheng\footnote{
	Institute of Physics and Korteweg-de Vries Institute for Mathematics,
	University of Amsterdam, Amsterdam, the Netherlands.
	On leave from CNRS, France.
	{\em E-mail:} {\tt mcheng@uva.nl}
		}\\
	John F. R. Duncan\footnote{
         Department of Mathematics and Computer Science,
         Emory University,
         Atlanta, GA 30322,
         U.S.A.
         {\em E-mail:} {\tt john.duncan@emory.edu}
                  }\\
	Jeffrey A. Harvey\footnote{
	Enrico Fermi Institute and Department of Physics,
         University of Chicago,
         Chicago, IL 60637,
         U.S.A.
         {\em E-mail:} {\tt j-harvey@uchicago.edu}
          }
}

\date{}

\maketitle
\abstract{
We analyze holomorphic Jacobi forms of weight one with level.
One such form plays an important role in umbral moonshine, leading to simplifications of the statements of the umbral moonshine conjectures. We prove that non-zero holomorphic Jacobi forms of weight one do not exist for many combinations of index and level, and use this to establish a characterization of the McKay--Thompson series of umbral moonshine in terms of Rademacher sums.
}

\clearpage

\tableofcontents

\clearpage

\section{Introduction}\label{sec:intro}

Umbral moonshine \cite{UM,MUM} attaches distinguished vector-valued mock modular forms to automorphisms of Niemeier lattices. To be specific, let $X$ be the root system of a Niemeier lattice (i.e. any self-dual even positive definite lattice of rank $24$, other than the Leech lattice). Define $G^X:=\Aut(N^X)/W^X$, where $N^X$ is the self-dual lattice associated to $X$ by Niemeier's classification \cite{Nie_DefQdtFrm24} (cf. also \cite{MR1662447}), and $W^X$ is the subgroup of $\Aut(N^X)$ generated by reflections in root vectors. Then \cite{MUM} describes an assignment $g\mapsto H^X_g$ of vector-valued holomorphic functions---the umbral McKay--Thompson series---to elements $g\in G^X$. (A very explicit description of this assignment appears in \S B of \cite{umrec}.)

The situation is analogous to monstrous moonshine \cite{MR554399}, where holomorphic functions $T_m$---the monstrous McKay--Thompson series---are attached to monster elements $m\in \MM$. In this case the $T_m$ are distinguished in that they are the normalized principal moduli (i.e. normalized hauptmoduls) attached to genus zero groups $\Gamma_m<\SL_2(\RR)$. Thanks to work \cite{borcherds_monstrous} of Borcherds, we know that they are also the graded trace functions arising from the action of $\MM$ on the graded infinite-dimensional $\MM$-module $V^{\natural}=\bigoplus_{n\geq -1}V^{\natural}_n$ constructed by Frenkel--Lepowsky--Meurman \cite{FLMPNAS,FLMBerk,FLM}.

Conjectures are formulated in \S6 of \cite{MUM}, in order to identify analogues for the $H^X_g$ of these two properties of the monstrous McKay--Thompson series. For the analogue of the normalized principal modulus property (also known as the genus zero property of monstrous moonshine), the notion of {optimal growth} is formulated in 
\cite{MUM}, following the work \cite{Dabholkar:2012nd} of Dabholkar--Murthy--Zagier.

Recall from \S6.3 of \cite{MUM} that a vector-valued function $H=(H_r)$ is called a mock modular form of {\em optimal growth} for $\G_0(n)$ with multiplier $\nu$, weight $\frac12$ and shadow $S$ if
\begin{gather}\label{eqn:conj:moon:opt}
	\begin{split}
	&\text{(i)}\quad H|_{\nu,\frac12,S}\g=H\text{ for all $\g\in \G_0(n)$},\\
	&\text{(ii)}\quad q^{\frac{1}{4m}}H_r(\t)=O(1)\text{ as $\t\to i\infty$ for all $r$,}\\
	&\text{(iii)}\quad H_r(\t)=O(1)\text{ for all $r$ as $\t\to \alpha\in \QQ$, whenever $\infty\notin \G_0(n)\alpha$}.
	\end{split}
\end{gather}
In condition (i) of (\ref{eqn:conj:moon:opt}) we write $|_{\nu,\frac12,S}$ for the weight $\frac12$ action of $\G_0(n)$ with multiplier $\nu$ and {\em twist} by $S$, on holomorphic vector-valued functions on the upper-half plane,
\begin{gather}
	\left(H|_{\nu,\frac12,S}\g\right)(\t)
	:=
	H(\g\t)\nu(\g)
	(c\t+d)^{-\frac12}
	+e(-\tfrac{1}{8})
	\int_{-\g^{-1}\infty}^{\infty}(z+\t)^{-\frac12}\overline{S(-\bar{z})}{\rm d}z
\end{gather}
for $\g=\left(\begin{smallmatrix}a&b\\c&d\end{smallmatrix}\right)$. Note that $e(x):=e^{2\pi i x}$. Roughly, the content of (i) of (\ref{eqn:conj:moon:opt}) is that $H$ is a mock modular form for $\Gamma_0(n)$. The other two conditions strongly restrict the growth of the components of $H$ near cusps.

Conjecture 6.5 of \cite{MUM} predicts that 
$H^X_g$ is the unique, up to scale, mock modular form of weight $\frac12$ for $\G_0(n)$ with optimal growth, for suitably chosen $n$, multiplier system and shadow. This conjecture is known to be true in many cases. For example, it holds for all $g\in G^X$, for $X=A_1^{24}$, as a consequence of the results of \cite{Cheng2011}. The validity of the conjecture for $g=e$, for all Niemeier root systems $X$, is proved in \cite{MUM}. (See Corollary 4.2, Proposition 4.6 and Proposition 5.2 in \cite{MUM}.)

But we now appreciate that the conjecture is false in general, 
due to the existence of non-zero holomorphic Jacobi forms of weight one. To be precise, for positive integers $m$ and $N$ say that a holomorphic function $\xi:\HH\times\CC\to \CC$ is a {\em (holomorphic) Jacobi form} of weight $1$, index $m$ and level $N$ if it is invariant for the usual weight $1$, index $m$ action of $\G^J_0(N):= \G_0(N)\ltimes\ZZ^2$, 
\begin{gather}
	\begin{split}\label{eqn:jac-jacact}
	(\xi|_{1,m}(\left(\begin{smallmatrix}a&b\\c&d\end{smallmatrix}\right),(0,0)))(\tau,z)&:=
	\xi\left(\frac{a\tau+b}{c\tau+d},\frac{z}{c\tau+d}\right)e\left(-m\frac{cz^2}{c\tau+d}\right)\frac{1}{c\tau+d},\\
	(\xi|_{1,m}(\left(\begin{smallmatrix}1&0\\0&1\end{smallmatrix}\right),(\lambda,\mu)))(\tau,z)&:=
	\xi\left(\tau,z+\lambda\tau+\mu\right)e\left(m\lambda^2\tau+2m\lambda z\right),
	\end{split}
\end{gather}
and if every {\em theta-coefficient} $h_r$ in the {\em theta-decomposition}
\begin{gather}\label{eqn:jac-thtdec}
	\xi(\t,z)=\sum_{r\xmod 2m}h_r(\t)\theta_{m,r}(\tau,z)
\end{gather}
remains bounded near every cusp of $\G_0(N)$. Say that $\xi$ is a Jacobi {\em cusp form} if the $h_r$ all tend to $0$, near all cusps of $\G_0(N)$.
In 
(\ref{eqn:jac-thtdec}) we write 
$\theta_{m,r}$ for the theta functions 
\begin{gather}\label{eqn:jac-thmr}
\theta_{m,r}(\tau,z):=\sum_{k\in\ZZ}q^{\frac1{4m}(2km+r)^2}y^{2km+r}
\end{gather}
attached to the even lattice $\sqrt{2m}\ZZ$, where $q=e(\tau)$ and $y=e(z)$. 
For later use we define 
\begin{gather}\label{eqn:jac-thmrpm}
\theta_{m,r}^\pm(\tau,z):=\theta_{m,-r}(\tau,z)\pm\th_{m,r}(\tau,z),
\end{gather}
and note that $\th_{m,r}(h\tau,hz)=\th_{mh,rh}(\tau,z)$ for $h$ a positive integer.

Write $J_{1,m}(N)$ for the space of holomorphic Jacobi forms of weight $1$, index $m$ and level $N$. 
Skoruppa has proved \cite{Sko_Thesis} that $J_{1,m}(1)=\{0\}$ for all $m$. More recently, an argument of Schmidt \cite{MR2545599}, demonstrates that 
there are no holomorphic Jacobi forms of weight $1$, for any index and level. Thus it was a surprise to us when we discovered\footnote{We are grateful to Ken Ono for pointing out to us that non-zero holomorphic weight one Jacobi forms should exist.} some non-zero examples. 

\begin{prop}\label{prop:jac:holexist}
Non-zero holomorphic weight one Jacobi forms exist. 
\end{prop}

Indeed, examples may be extracted from the existing literature. 
Recall the Dedekind eta function, $\eta(\t):=q^{\frac1{24}}\prod_{n>0}(1-q^n)$. 
The function $\eta(\t)\th^-_{2,1}(\t,\frac12z)$ 
appears in \S1.6, Example 1.14, of \cite{GriNik_AutFrmLorKMAlgs_II}. (Note that $\th_{2,1}^-(\tau,\frac12z)$ is $-\vartheta(\tau,z)$ in loc. cit.) 
It is a Jacobi form of weight $1$, index $\frac12$ and level $1$, with a non-trivial multiplier system. 
Observe that if $h$ is a positive integer and $\xi(\tau,z)$ is a Jacobi form of index $m$ and level $N$ with some multiplier system, 
then $\xi(\tau,hz)$ is a Jacobi form of index $mh^2$ with the same level, and $\xi(h\tau,hz)$ is a Jacobi form of index $mh$ and level $Nh$, and in both cases the weight is unchanged and the multiplier system transforms in a controlled way. 
From the explicit description of the multiplier system of $\eta(\tau)\th_{2,1}^-(\tau,\frac12 z)$ in \cite{GriNik_AutFrmLorKMAlgs_II}, 
we deduce that 
\begin{gather}
\xi_{1,12}(\t,z):=
\eta(6\t)\th^-_{2,1}(6\t,6z)=
(y^{-6}-y^6)q+O(q^7)
\end{gather} 
is a non-zero element of $J_{1,12}(36)$. 
This proves Proposition \ref{prop:jac:holexist}. Indeed, it proves the a priori stronger statement, that weight one cusp forms exist, because it is easily checked that $\eta(6\t)$ vanishes at all cusps of $\Gamma_0(36)$. In a similar way we obtain a non-zero, non-cuspidal 
element 
\begin{gather}\label{eqn:intro-xi18}
\xi_{1,8}(\t,z):=\th_{8,4}(\t,0)\th^-_{8,4}(\t,z)=(y^{-4}-y^4)q+O(q^5)
\end{gather}
of $J_{1,8}(32)$, by considering the function 
${\eta(2\t)^2}{\eta(\t)^{-1}}\th^-_{2,1}(\t,\frac12z)$ which appears in Theorem 1.4 of \cite{MR2806099}. 
(Note that $\th_{8,4}(\t,0)=\eta(8\t)^2\eta(4\t)^{-1}$.)

An infinite family of examples 
may be extracted from 
Corollary 4.9 of 
\cite{MR3123592}. For $a,b\in\ZZ^+$ define the {\em theta quark}\footnote{Theta quarks belong to a more general theory of {\em theta blocks} due to Gritsenko--Skoruppa--Zagier \cite{GSZ_ThtBlk}.}
\begin{gather}\label{eqn:thqk}
Q_{a,b}(\tau,z):= \eta(\tau)^{-1}{\th_{2,1}^-(\tau,\tfrac 12az)\th_{2,1}^-(\tau,\tfrac 12bz)\th_{2,1}^-(\tau,\tfrac 12(a+b)z)}.
\end{gather}
Then $Q_{a,b}$ (denoted $\th_{a,b}$ in \cite{MR3123592}) is a Jacobi form of weight $1$, index $m=a^2+ab+b^2$ and level $N=1$, with multiplier system satisfying 
\begin{gather}\label{eqn:Qabxform}
Q_{a,b}|_{1,m}(\left(\begin{smallmatrix}1&1\\0&1\end{smallmatrix}\right),(\lambda,\mu))=e(\tfrac13)
Q_{a,b},\quad
Q_{a,b}|_{1,m}(\left(\begin{smallmatrix}0&-1\\1&0\end{smallmatrix}\right),(\lambda,\mu))=Q_{a,b},
\end{gather} 
for $\lambda,\mu\in\ZZ$. From (\ref{eqn:Qabxform}) we deduce that $Q_{a,b}(3\tau,3z)$ is a non-zero element of $J_{1,3m}(9)$ for $m=a^2+ab+b^2$, for any $a,b\in \ZZ^+$. The first of these theta-quarks, $Q_{1,1}$, will play an important role in the sequel. 
To set up for this we define 
\begin{gather}
	\begin{split}
	\label{eqn:jac:f9s}
	\xi^{(9)}_{3A}(\tau,z)&:=\theta_{3,3}(\tau,0)\th^-_{9,3}(\tau,z)-\theta_{3,0}(\tau,0)\th^-_{9,6}(\tau,z),\\
	\xi^{(9)}_{6A}(\tau,z)&:=\theta_{3,3}(\tau,0)\th^-_{9,3}(\tau,z)+\theta_{3,0}(\tau,0)\th^-_{9,6}(\tau,z).
	\end{split}
\end{gather}
The first of these is $Q_{1,1}(3\tau,3z)$ and belongs to $J_{1,9}(9)$, the second belongs to $J_{1,9}(36)$. 

Despite contradicting Proposition \ref{prop:jac:holexist} it seems that the argument of \cite{MR2545599} is correct, except that it depends upon a result from \cite{MR2379341} which is stated more strongly than what is proven in that paper. It is apparently this overstatement that led to the false conclusion of \cite{MR2545599}, that $J_{1,m}(N)$ vanishes for all $m$ and $N$. In an erratum \cite{IbuSko_Cor} to \cite{MR2379341}, it is proven that $J_{1,m}(N)$ does vanish whenever $m$ and $N$ are coprime and $N$ is square-free. In fact, the following more general result is obtained.
\begin{thm}[\!\!\cite{IbuSko_Cor}]\label{thm:ibuskocor}
Let $m$ and $N$ be positive integers and assume that $N$ is square-free. Let $N_m$ be the smallest divisor of $N$ such that $m$ and $\frac{N}{N_m}$ are relatively prime. If $\frac{N}{N_m}$ is odd then $J_{1,m}(N)=J_{1,m}(N_m)$. If $\frac{N}{N_m}$ is even then $J_{1,m}(N)=J_{1,m}(2N_m)$. Also, if $m$ is odd then $J_{1,m}(2)=\{0\}$.
\end{thm}
The fact that $J_{1,m}(N)=\{0\}$ when $N$ is square-free and $(m,N)=1$ is obtained by taking Theorem \ref{thm:ibuskocor} together with Skoruppa's earlier result \cite{Sko_Thesis} that $J_{1,m}(1)=\{0\}$ for all $m$.

In this article we revisit the conjectures of \cite{MUM}, in light of the existence of holomorphic weight one Jacobi forms, and we establish the vanishing of $J_{1,m}(N)$ for many $m$ and $N$. 
As we explain in \S\ref{sec:conj:mod}, the particular Jacobi forms (\ref{eqn:intro-xi18}) and (\ref{eqn:jac:f9s}) play special roles in moonshine, the latter leading us to a simplification (Conjecture \ref{conj:conj:mod:Kell}) of the umbral moonshine module conjecture. It is this simplified form of the umbral moonshine module conjecture which is described in the recent review \cite{mnstmlts}, and proven in \cite{umrec}. 

In \S\ref{sec:conj:rad} we formulate a characterization (Conjecture \ref{conj:conj:moon:rad}) of the umbral McKay--Thompson series in terms of Rademacher sums. This serves as a natural analogue of the genus zero property of monstrous moonshine, and a replacement for Conjecture 6.5 of \cite{MUM} which is now known to be false in general. We 
point out a geometric interpretation of the theta quark $Q_{1,1}$ 
in \S\ref{sec:conj:pas}, and also indicate a possible connection to physics.

The remainder of the paper is devoted to a proof of the Rademacher sum characterization, Conjecture \ref{conj:conj:moon:rad}. This follows (Corollary \ref{cor:wt1-rad}) from our main result, Theorem \ref{thm:wt1-main}, which states that a certain family of holomorphic Jacobi forms of weight one vanishes identically. The proof of Theorem \ref{thm:wt1-main} is given in \S\ref{sec:wt1:pmr}. We obtain it by applying a representation theoretic approach that has been developed by Skoruppa (cf. \cite{MR2512363}). We review relevant background in \S\ref{sec:wt1:weil} and \S\ref{sec:wt1:ppp}. The methods of \S\ref{sec:wt1:pmr} work best for levels that are not divisible by high powers of $2$ or $3$. In \S\ref{sec:wt1:exp} we use a more elementary technique to prove (mostly) complementary vanishing results for some spaces with small index.

We attach four appendices to this article. In \S\ref{sec:coeffs} we present revised coefficient tables for the umbral McKay--Thompson series $H^X_g$ with $X=A_8^3$. Using these together with the character table of $G^X$ in \S\ref{sec:chars} one can compute the (revised) multiplicities of irreducible $G^X$-modules that appear in \S\ref{sec:decompositions}. The proof of our main result (Theorem \ref{thm:wt1-main}) is obtained by applying the lemmas we prove in \S\ref{sec:wt1:pmr} to the tables in \S\ref{sec:levels}.

\section{Umbral Moonshine}\label{sec:conj}

\subsection{Umbral Moonshine Modules}\label{sec:conj:mod}

Conjecture 6.1 of \cite{MUM} predicts the existence of umbral analogues $K^X$ of $V^{\natural}$, for each Niemeier root system $X$. At the time this conjecture was formulated we had identified candidate functions $H^{X}_{g,r}$, for $g\in G^X$, which implied that $K^X$ would have peculiar properties in the case that $X=A_8^3$. Nonetheless, we did not think that there could be other possibilities for the $H^X_{g,r}$ with $X=A_8^3$, in light of Schmidt's result \cite{MR2545599} discussed above. 

It turns out that $\xi^{(9)}_{g}(\tau,z)$ (cf. (\ref{eqn:jac:f9s})) and $\phi^X_{g}(\tau,z):=\sum_{r\xmod 18}H^X_{g,r}(\t)\th_{9,r}(\tau,z)$ have the same multiplier system, for $X=A_8^3$ and $o(g)=0\xmod 3$, on $\Gamma_0(3)$ for $g\in 3A$, and on $\Gamma_0(6)$ in the case that $g\in 6A$. Thus we may use these holomorphic Jacobi forms to revise our prescription for $H^X_{g}$, for $X=A_8^3$, for the two conjugacy classes $[g]\subset G^X$ with $o(g)=0\xmod 3$. 

Doing this we are led to a reformulation of Conjecture 6.1 of \cite{MUM} that is uniform with respect to the choice of Niemeier root system $X$. To be precise, let us redefine $H^X_{3A}$ and $H^X_{6A}$, for $X=A_8^3$, by setting 
\begin{gather}
H^X_{g}(\t):=H^{X,{\rm old}}_g(\t)+t^{(9)}_{g}(\t), 
\end{gather}
where $t^{(9)}_g=(t^{(9)}_{g,r})$ is the vector-valued theta series of weight $\frac12$, defined by
\begin{gather}\label{eqn:jac:t9s}
	t^{(9)}_{3A,r}(\t):=
	\begin{cases}
		0,&\text{ if $r=0\xmod 9$, or $r\neq 0\xmod 3$,}\\
		\mp\theta_{3,3}(\tau,0),&\text{ if $r\pm 3\xmod 18$,}\\
		\pm\theta_{3,0}(\tau,0),&\text{ if $r\pm 6\xmod 18$,}\\
	\end{cases}\\
	t^{(9)}_{6A,r}(\t):=
	\begin{cases}
		0,&\text{ if $r=0\xmod 9$, or $r\neq 0\xmod 3$,}\\
		\mp\theta_{3,3}(\tau,0),&\text{ if $r\pm 3\xmod 18$,}\\
		\mp\theta_{3,0}(\tau,0),&\text{ if $r\pm 6\xmod 18$.}\\
	\end{cases}
\end{gather}
Note that $t^{(9)}_g$ is just the vector-valued modular form whose components are the theta coefficients of $\xi^{(9)}_g$ (cf. (\ref{eqn:jac:f9s})).

We present the corresponding coefficient tables (revisions of Tables 75--82 in \cite{MUM}) in \S\ref{sec:coeffs}. This leads to revised decomposition tables (revisions of Tables 170 and 171 in \cite{MUM}), which we present in \S\ref{sec:decompositions}. Motivated by these new decompositions we reformulate Conjecture 6.1 of \cite{MUM} as follows.

\begin{conj}\label{conj:conj:mod:Kell}
Let $\rs$ be a Niemeier root system and let ${ m}$ be the Coxeter number of $\rs$. There exists a naturally defined $\ZZ/2m\ZZ\times\QQ$-graded super-module
\begin{gather}
	K^{\rs}=\bigoplus_{r\text{ mod }2m} K^{\rs}_r=\bigoplus_{r\text{ mod }2m}
	\bigoplus_{\substack{D\in\ZZ\\D=r^2\text{ mod }4m}}
	K^{\rs}_{r,-\frac{D}{4m}}
\end{gather}
for $G^X$, such that 
the graded super-trace attached to an element $g\in G^{\rs}$ is recovered from the vector-valued mock modular form $H^{\rs}_g$ via  
\begin{gather}\label{eqn:conj:mod:str}
	H^{\rs}_{g,r}(\tau)=
	\sum_{\substack{D\in\ZZ\\D=r^2\text{ mod }4m}}\str_{K^{\rs}_{r,-\frac{D}{4m}}}(g)\,q^{-\frac{D}{4m}}.
\end{gather}
Moreover, $K^X_r=\{0\}$ for $r=m\xmod 2m$. If $0<r<m$, then the homogeneous component $K^{\rs}_{r,d}$ of $K^{\rs}$ is purely even for $d\geq 0$, and purely odd for $d<0$. If $-m<r<0$, then the homogeneous component $K^{\rs}_{r,d}$ is purely odd for $d\geq 0$, and purely even for $d<0$. 
\end{conj}

The existence of $K^X$ for $X=A_1^{24}$ was established by Gannon in \cite{MR3539377}. More recently, the existence of all the $G^X$-modules $K^X$ satisfying the specifications of Conjecture \ref{conj:conj:mod:Kell} has been established in \cite{umrec}.

The formulation of Conjecture \ref{conj:conj:mod:Kell} that appears in \cite{umrec} (cf. also \S9.3 of \cite{mnstmlts}) is slightly different from the above, avoiding the use of superspaces and supertraces. We now explain the equivalence.

Recall that $H^X_{g,r}=-H^{X}_{g,-r}$ for all $X$, $g\in G^X$ and $r\in\ZZ/2m\ZZ$, so 
to prove Conjecture \ref{conj:conj:mod:Kell} it suffices to construct the $K^{\rs}_{r}$ for $0<r<m$. If the highest rank irreducible component of $\rs$ is not of type A then there are further symmetries amongst the $H^X_{g,r}$ that further reduce the problem. Namely, if $X$ has a type D component but no type A components then $m=2\xmod 4$, we have $H^X_{g,r}=0$ for $r=0\xmod 2$, and $H^{X}_{g,r}=H^{X}_{g,m-r}$. So it suffices to consider the $K^{\rs}_r$ for $r\in\{1,3,5,\ldots,\frac12 m\}$. If $X$ has no components of type A or D then either $X=E_6^4$ or $X=E_8^3$. For $X=E_6^4$ it suffices to consider $r\in\{1,4,5\}$, and for $X=E_8^3$ we need only $r\in\{1,7\}$.

With this in mind let us define $I^X\subset\ZZ/2m\ZZ$ by setting $I^X:=\{1,2,3,\ldots,m-1\}$ in case $X$ has a type A component. Set $I^X:=\{1,3,5,\ldots,\frac12 m\}$ in case $X$ has a type D component but no type A components, and set $I^X:=\{1,4,5\}$ for $X=E_6^4$, and $I^X:=\{1,7\}$ for $X=E_8^3$.

Define $\check{H}^X_{g}(\t)$ to be the $I^X$-vector-valued function (i.e. vector-valued function with components indexed by $I^X$) whose components are the $H^X_{g,r}$ for $r\in I^X$. The $\check{H}^X_g$ inherit good modular properties from the $H^X_g$. For example, if $X$ has a simple component $A_{m-1}$ then the multiplier system of $\check{H}^X_e$ is the inverse of that associated to 
\begin{gather}\label{eqn:checkSAm-1}
	\check{S}^{A_{m-1}}(\t):=(S_{m,1}(\t),\ldots,S_{m,m-1}(\t)), 
\end{gather}
where $S_{m,r}(\tau):=\sum_{k\in\ZZ}(2km+r)q^{\frac1{4m}(2km+r)^2}$. For $X=E_8^3$ the multiplier system of $\check{H}^X_e$ is the inverse of that associated to 
\begin{gather}
	\check{S}^{E_{8}}(\t):=((S_{30,1}+S_{30,11}+S_{30,19}+S_{30,29})(\t),(S_{30,7}+S_{30,13}+S_{30,17}+S_{30,23})(\t)). 
\end{gather}
(Cf. \S4.1 of \cite{MUM}.) More generally, if the mock modular form $H^X_g$ transforms under $\Gamma_0(n_g)$ with multiplier system $\nu^X_g$ and shadow $S^X_g$ then there is a directly related multiplier system $\check{\nu}^X_g$ 
such that $\check{H}^X_g$ transforms under $\Gamma_0(n_g)$ with multiplier $\check{\nu}^X_g$ and shadow $\check{S}^X_g$, where $\check{S}^X_g$ is the $I^X$-vector-valued cusp form whose components are the $S^X_{g,r}$ for $r\in I^X$.

Note that only one component of $\check{H}^X_{g}$ has a pole. Namely, 
\begin{gather}
\check{H}^X_{g,1}(\tau)=H^X_{g,1}(\t)=-2q^{-\frac{1}{4m}}+O(q^{1-\frac{1}{4m}}). 
\end{gather}
In terms of the $\check{H}^X_g$, Conjecture \ref{conj:conj:mod:Kell} may be rephrased 
as the statement that there exist naturally defined bi-graded $G^X$-modules
\begin{gather}
	\check{K}^{\rs}=
	\bigoplus_{r\in I^X} \check{K}^{\rs}_r=\bigoplus_{r\in I^X}
	\bigoplus_{\substack{D\in\ZZ\\D=r^2\text{ mod }4m}}
	\check{K}^{\rs}_{r,-\frac{D}{4m}}
\end{gather}
such that 
the graded trace attached to an element $g\in G^{\rs}$ is recovered from the vector-valued mock modular form $\check{H}^{\rs}_g$ via  
\begin{gather}\label{eqn:conj:mod:strtilde}
	\check{H}^{\rs}_{g,r}(\tau)=-2q^{-\frac{1}{4m}}\delta_{r,1}+
	\sum_{\substack{D\in\ZZ\\D=r^2\text{ mod }4m}}\tr_{\check{K}^{\rs}_{r,-\frac{D}{4m}}}(g)\,q^{-\frac{D}{4m}}.
\end{gather}
This is the form in which Conjecture \ref{conj:conj:mod:Kell} has been expressed in \cite{umrec,mnstmlts}.

Conjecture 6.11 from \cite{MUM} also concerns umbral moonshine modules. Now we may simplify it by removing the last sentence, concerning $X=A_8^3$, as the newly defined $H^X_g$ lead (conjecturally) to representations $K^X_{r,-\frac{D}{4m}}$ that are doublets for $G^X$, whenever $D\neq 0$. This is what we would expect, given the discussion in \S6.4 of \cite{MUM}, since there are no values $n$ attached to $X=A_8^3$ in Table 10 of \cite{MUM}. 

To conclude this section we note that the example (\ref{eqn:intro-xi18}) also plays a special role in moonshine, 
because it develops \cite{MR3108775} that $\xi_{1,8}(\frac14 \tau,\frac14 z)=\th_{2,1}(\tau,0)\th_{2,1}^-(\t,z)$ is the only non-zero Jacobi form appearing in generalised Mathieu moonshine that is not related to one of the Mathieu moonshine forms by the action by $\SL_2(\ZZ)$. A few further holomorphic Jacobi forms appear in generalised umbral moonshine; they are described explicitly in  Table 2 of \cite{Cheng:2016nto}. Conjecturally \cite{Cheng:2016nto} there are infinite-dimensional modules for certain deformations of the Drinfel'd doubles of the umbral groups $G^X$ that underly the functions of generalised umbral moonshine.

\subsection{Umbral Mock Modular Forms}\label{sec:conj:rad}

We require to modify the statement of Conjecture 6.5 of \cite{MUM}, in light of the discussion in \S\ref{sec:intro}. Since the notion of optimal growth is too weak to determine the umbral McKay--Thompson series in general, we recast our reformulation in terms of Rademacher sums, thus generalizing Conjecture 5.4 of \cite{UM}. The $\check{H}^X_g$ are well-adapted to this, as they each have a pole in exactly one component.

Write $\Gamma_{\infty}$ for the group of upper-triangular matrices in $\SL_2(\ZZ)$. For $\alpha\in \RR$ and $\g\in \SL_2(\ZZ)$ 
define $\rad^{[\alpha]}_{\frac12}(\gamma,\tau):=1$ if $\gamma\in \G_\infty$. For $\gamma=\left(\begin{smallmatrix}a&b\\c&d\end{smallmatrix}\right)$ not in $\Gamma_\infty$ set
\begin{gather}
\rad^{[\alpha]}_{\frac12}(\left(\begin{smallmatrix}a&b\\c&d\end{smallmatrix}\right),\tau):=
	e\left(\tfrac{\a}{c(c\tau+d)}\right)
		\sum_{k\geq 0}
\frac{\left(-2\pi i\tfrac{ \alpha}{ c(c\tau+d)}\right)^{n+\frac12}}{\Gamma(n+\tfrac32)},
\end{gather}
where $e(x):=e^{2\pi i x}$ and we use the principal branch to define $z^\frac12$ for $z\in \CC$.
Suppose that $\nu$ is a multiplier system for vector-valued modular forms of weight $\frac12$ on $\Gamma=\Gamma_0(n)$, for some $n$, and suppose that $\nu=(\nu_{ij})$ satisfies $\nu_{11}(\left(\begin{smallmatrix}1&1\\0&1\end{smallmatrix}\right))=e(\frac14m)$, for some basis $\{\gt{e}_i\}$, for some positive integer $m$. 
To this data we attach the {\em Rademacher sum}
\begin{gather}\label{eqn:conj:Rad}
	R_{\Gamma,\nu}(\tau)
	:=
	\lim_{K\to \infty}
	\sum_{\left(\begin{smallmatrix}a&b\\c&d\end{smallmatrix}\right)\in\G_\infty\backslash\G_{K,K^2}}
	{\nu}(\left(\begin{smallmatrix}a&b\\c&d\end{smallmatrix}\right))e\left(-\tfrac{1}{4m}\tfrac{a\tau+b}{c\tau+d}\right)\gt{e}_1
	{(c\tau+d)^{-\frac12}}
	\rad^{[-\frac{1}{4m}]}_{\frac12}(\left(\begin{smallmatrix}a&b\\c&d\end{smallmatrix}\right),\tau)
	,
\end{gather}
where $\Gamma_{K,K^2}:=\left\{\left(\begin{smallmatrix}a&b\\c&d\end{smallmatrix}\right)\in\Gamma\mid 0\leq c<K,\,|d|<K^2\right\}$. 
See \cite{Cheng:2014fk} for an introduction to Rademacher sums, and \cite{2014arXiv1406.0571W} for a general and detailed discussion of the vector-valued case.

We now use the 
construction (\ref{eqn:conj:Rad}) to formulate a replacement for Conjecture 6.5 of \cite{MUM}. In preparation for the case that $X=A_8^3$ define $\check{t}^{(9)}_g$ to be the $8$-vector-valued theta series whose components are the $t^{(9)}_{g,r}$ for $0<r<9$ (cf. (\ref{eqn:jac:t9s})).
\begin{conj}\label{conj:conj:moon:rad}
Let $X$ be a Niemeier root system and let $g\in G^X$. If $X\neq A_8^3$ and $g\in G^X$, or if $X=A_8^3$ and $g\in G^X$ does not satisfy $o(g)=0\xmod 3$, then we have
\begin{gather}
	\check{H}^X_{g}(\t)=-2R^X_{\Gamma_0(n_g),\check{\nu}^X_g}(\tau).
\end{gather}
If $X=A_8^3$ and $g\in G^X$ satisfies $o(g)=0\xmod 3$ then
\begin{gather}
	\check{H}^X_{g,r}(\t)=-2R^X_{\Gamma_0(n_g),\check{\nu}^X_g}(\t)+\check{t}^{(9)}_{g}(\t).
\end{gather}
\end{conj}
According to the discussion of \S5.2 of \cite{UM}, Conjecture \ref{conj:conj:moon:rad} is a natural analogue of the genus zero property of monstrous moonshine. It also naturally generalizes Conjecture 5.4 of \cite{UM}. Evidently the case $X=A_8^3$ requires special treatment from the point of view of Rademacher sums, but the difference between $\check{H}^X_{g}$ and $-2R_{\Gamma_0(n_g),\check{\nu}^X_g}^X$ for $X=A_8^3$ is slight, for the coefficients of $\check{t}^{(9)}_g$ are bounded, and almost always\footnote{The coefficients of $\check{t}_g^{(9)}$ are supported on perfect square exponents, so asymptotically, 100\% of them vanish.} zero. 

\subsection{Paramodular Forms}\label{sec:conj:pas}

The content of \S\ref{sec:conj:mod} and \S\ref{sec:conj:rad} demonstrate the importance of the theta quark $Q_{1,1}$ (cf. (\ref{eqn:thqk})) to umbral moonshine. In this short section we 
point out a relation between $Q_{1,1}$ 
and the geometry of complex surfaces, and a possible connection to physics.

To prepare for this recall the {\em (degree $2$) Siegel upper half-space}, defined by 
\begin{gather}
\HH_2:=\left\{
Z= \begin{pmatrix} \tau & z \\ z & \sigma \end{pmatrix}\in M_2(\CC)
\mid
\Im(Z)>0
\right\},
\end{gather}
which is acted on naturally by the symplectic group $\Sp_4(\RR)$. For $t$ a positive integer  
define the {\em paramodular group} $\G_t<\Sp_4(\QQ)$ by setting
\begin{gather}
	\G_t:=
	\left\{
	\begin{pmatrix}
		*&	*t&	*&	*\\
		*&	*&	*&	\frac1t*\\
		*&	*t&	*&	*\\
		*t&	*t&	*t&	*
	\end{pmatrix}
	\in\Sp_4(\QQ)\mid \text{ all $*$ in $\ZZ$}
	\right\}.
\end{gather}
Then $\cA_t:=\Gamma_t\backslash\HH_2$ is a coarse moduli space for $(1,t)$-polarized abelian surfaces (cf. \cite{MR1277050}, where $\Gamma_t$ is denoted $\Gamma[t]$).

For $k$ an integer the {weight $k$} action of $\G_t$ on functions $F:\HH_2\to\CC$ is defined by setting
\begin{gather}\label{eqn:para:actn}
	(F|_k\g)(Z):={\det(CZ+D)^{-k}}F\left({(AZ+B)}{(CZ+D)^{-1}}\right)
\end{gather}
for $\g=\left(\begin{smallmatrix}A&B\\C&D\end{smallmatrix}\right)\in\G_t$. 
Holomorphic functions that are invariant for this action are 
called {\em paramodular forms} of weight $k$ for $\Gamma_t$, or {\em Siegel modular forms} in case $t=1$. Write $M_k(\G_t)$ for the vector space they comprise. Then the space of holomorphic sections of the canonical line bundle on $\cA_t$ may be identified with $M_3(\G_t)$. According to Hilfsatz 3.2.1 of \cite{MR0469872} a holomorphic differential on $\cA_t$ represented by $F\in M_3(\G_t)$ extends to a holomorphic differential on a non-singular model of a compactification of $\cA_t$ if and only if $F$ is cuspidal.

Suppose $F\in M_k(\G_1)$ is a Siegel modular form. Then setting $p=e(\s)$ and writing $F(Z)=\sum_{m\geq 0}\phi_m(\t,z)p^m$ it follows from the invariance of $F$ under the action (\ref{eqn:para:actn}) that $\phi_m$ is 
a Jacobi form\footnote{This is one of the main motivations for the notion of Jacobi form. Cf. \cite{feingold_frenkel} for an early analysis together with applications to hyperbolic Kac--Moody Lie algebras.} of weight $k$ and index $m$. In particular, $\phi_1$ has index $1$. Maass discovered \cite{MR582704} a lifting $J_{k,1}\to M_k(\Gamma_1)$ which reverses this process, assigning a Siegel modular form $F\in M_k(\G_1)$ to a Jacobi form $\phi\in J_{k,1}$ in such a way that $\phi$ is the coefficient of $p$ in $F$. Gritsenko introduced a generalization $J_{k,t}\to M_k(\G_t)$ of the Maass lift in \cite{MR1345176} (cf. also \cite{MR1277050,GriNik_AutFrmLorKMAlgs_II}), 
and a further generalization 
adapted to Jacobi forms with level and character appears as Theorem 2.2 in \cite{MR2806099}.

Applying Theorem 2.2 of \cite{MR2806099} to 
the 
theta quark $Q_{1,1}$ we obtain a paramodular form
\begin{gather}
	X_{1,9}(Z):=\sum_{\substack{m>0\\m=1\xmod 3}}\widetilde Q_{1,1}|_1T^{(1)}_-(m)(Z)
\end{gather}
of weight $1$ for $\G_9$ with a character of order $3$. 
(See loc. cit. for the operators $\phi\mapsto \widetilde\phi|_kT^{(N)}_-(m)$.) So the cube of $X_{1,9}$ defines a holomorphic differential on  $\cA_9$. 

It was shown by O'Grady \cite{MR1030139} (see also \cite{MR1827859}) that the Satake compactification of $\cA_9$ is rational, so there are no cusp forms of weight $3$ for $\G_9$. So $X_{1,9}^3$ is an example of a paramodular form of odd weight that is not a cusp form. 
It is in some sense a first example of a non-cuspidal paramodular form with odd weight, because it can be shown\footnote{We thank Cris Poor and David Yuen for explaining this.} by restriction to the $1$-dimensional cusps of $\Gamma_t$ that all forms in $M_k(\Gamma_t)$ are cuspidal when $k$ is odd and $t$ is neither divisible by $16$, nor divisible by the square of any odd prime. So all odd weight paramodular forms for $\Gamma_t$ are cuspidal if $t<9$. 

We refer to \cite{MR3031888} 
for a detailed analysis of the cusps of $\Gamma_t$. 
Theorem 8.3 of \cite{MR3283174} gives a more general construction of paramodular forms of weight $3$, 
and we can recover $X_{1,9}^3$ by taking $a_i=b_i=1$ there. 
Robust methods for computing spaces of paramodular forms precisely are developed and applied in \cite{MR3498287}.

It is interesting to note that $\mathcal{A}_t$ also appears as a moduli space describing massless degrees
of freedom in certain compactifications of heterotic string theory, and associated paramodular
forms have been shown 
\cite{LopesCardoso:1996zj,MR1668093,MR3158917} 
to govern one-loop corrections to their interaction terms. 
This may be a good setting in which to understand the physical significance of $Q_{1,1}$.
A discussion of umbral mock modular forms in this context appears in \S5.5 of \cite{UM}.

\section{Weight One Jacobi Forms}\label{sec:wt1}

In this section we prove our main result. 
To formulate it set $\phi^X_g:=\sum_r H^X_{g,r}\theta_{m,r}$ for $X$ a Niemeier root system and $g\in G^X$. In general $\phi^X_g$ is a (weak) mock Jacobi form of weight $1$ and index $m$, where $m=m^X$ is the Coxeter number of any simple component of $X$. We refer to \cite{Dabholkar:2012nd} or \cite{omjt} for background on mock Jacobi forms.

\begin{thm}\label{thm:wt1-main}
Let $X$ be a Niemeier root system and let $g\in G^X$. 
If $\xi$ is a holomorphic Jacobi form of weight $1$ and index $m^X$ 
with the same level and multiplier system as $\phi^X_g$ then $\xi=0$, except possibly if $X=A_8^3$ and $o(g)=0\xmod 3$.
\end{thm}
Theorem \ref{thm:wt1-main} has a direct bearing on Conjecture \ref{conj:conj:moon:rad}.
\begin{cor}\label{cor:wt1-rad}
Let $X$ be a Niemeier root system and let $g\in G^X$. If the Rademacher sum $R^X_{\Gamma_0(n_g),\check{\nu}_g^X}$ converges then the identity for $\check{H}^X_g$ predicted in Conjecture \ref{conj:conj:moon:rad} is true. 	
\end{cor}
A proof of the convergence of the Rademacher sums $R^X_{\Gamma_0(n_g),\check{\nu}_g^X}$ is given in \S 3 of \cite{umrec}, and closely related convergence results are given in \cite{MR3582425}. The $X=A_1^{24}$ case of Conjecture \ref{conj:conj:moon:rad} was proven first in \cite{Cheng2011}, via different methods.

\begin{proof}[Proof of Corollary \ref{cor:wt1-rad}]
Let $X$ be a Niemeier root system and let $g\in G^X$. 
To the Rademacher sum $R=R^X_{\Gamma_0(n_g),\check{\nu}_g^X}$ with components $R_r$ for $r\in I^X$ we associate a $2m$-vector-valued function $\hat R=(\hat R_r)$, with components indexed by $\ZZ/2m\ZZ$, as follows. For $X\notin\{E_6^4, E_8^3\}$ we set 
$\hat R_r:=\pm R_r$ for $\pm r\in I^X$, and $\hat R_r:=0$ for $\pm r \notin I^X$. 
In case $X=E_6^4$ set $R_7:=R_1$, $R_8:=R_4$, $R_{11}:=R_5$ and $\hat I^X:=\{1,4,5,7,8,11\}$, and then define $\hat R_r:=\pm R_r$ for $\pm r\in \hat I^X$, and $\hat R_r:=0$ for $\pm r \notin \hat I^X$. For $X=E_8^3$ set $R_{29}:=R_{19}:=R_{11}:=R_1$, $R_{23}:=R_{17}:=R_{13}:=R_{7}$, and $\hat I^X:=\{1,7,11,13,17,19,23,29\}$, and then define $\hat R_r$ in analogy with the case $X=E_6^4$.
Then $\rho:=\sum_r \hat R_r\theta_{m,r}$ is a weak mock Jacobi form of weight $1$ and index $m$ with the same level and multiplier system as $\phi^X_g$. Also, the polar parts of $-2\hat R$ and $H^X_g$ are the same by construction, so $\xi:=\phi^X_g+2\rho$ is a holomorphic mock Jacobi form of weight $1$ with some level. 

We claim that the shadow of $\xi$ must vanish, so that $\xi$ is actually a holomorphic Jacobi form. To see this let $\xi=\sum h_r\th_{m,r}$ be the theta-decomposition of $\xi$ and let $\sum \bar{g}_r\th_{m,r}$ be the shadow of $\xi$. Then $g=(g_r)$ is a $2m$-vector-valued cusp form of weight $\frac32$ transforming with the conjugate multiplier to that of $H^X_g=(H^X_{g,r})$. The proof of Proposition 3.5 in \cite{BruFun_TwoGmtThtLfts} shows that if $g'=(g_r')$ is another vector-valued cusp form of weight $\frac32$ transforming in the same way as $g$ then the Petersson inner product $\langle g,g'\rangle$ vanishes unless some of the $h_r$ have non-zero polar parts. But all the $h_r$ are bounded at all cusps by construction. So $g=(g_r)$ is orthogonal to all cusp forms, so vanishes identically. 

So $\xi$ is a holomorphic Jacobi form as claimed. Applying Theorem \ref{thm:wt1-main} to $\xi$ we conclude that $\xi$ vanishes identically unless $X=A_8^3$ and $o(g)=0\xmod 3$. So $\phi^X_g=-2\rho$, which is the prediction of Conjecture \ref{conj:conj:moon:rad}, for all $g\in G^X$ for all $X$, except when $X=A_8^3$ and $g$ belongs to the $3A$ or $6A$ conjugacy classes of $G^X$. To complete the proof of Conjecture \ref{conj:conj:moon:rad} we need to show that, for $X=A_8^3$ and $g\in 3A\cup 6A$, the $r=3$ and $r=6$ components of the Rademacher sum $R$ vanish identically. This follows from the fact that if $\nu=(\nu_{ij})_{0<i,j<9}$ denotes the multiplier system of the $8$-vector-valued modular form $\check{S}^{A_8}$ (cf. (\ref{eqn:checkSAm-1})), and if $\gamma\in \Gamma_0(3)$, then the matrix entries $\nu_{ij}(\gamma)$ vanish whenever $i=0\xmod 3$ and $j\neq 0\xmod 3$, or $j=0\xmod 3$ and $i\neq0 \xmod 3$. In other words, the multiplier of $\check{S}^{A_8}$ becomes block diagonal, with blocks indexed by $\{1,2,4,5,7,8\}$ and $\{3,6\}$, when restricted to $\Gamma_0(3)$. It follows from this, and the description of the multiplier systems for $3A$ and $6A$ in \cite{MUM}, or the explicit descriptions of the $H^X_g$ in \cite{umrec}, that the corresponding Rademacher sums have vanishing $r=3$ and $r=6$ components. The prediction of Conjecture \ref{conj:conj:moon:rad} follows, for $X=A_8^3$ and $o(g)=0\xmod 3$. This completes the proof.
\end{proof}

We prove Theorem \ref{thm:wt1-main} in \S\ref{sec:wt1:pmr}. To prepare for this 
we review the metaplectic double cover of the modular group, and some results from \cite{Sko_Thesis} in \S\ref{sec:wt1:weil}, and we review some facts about characters of Weil representations following \cite{MR2512363} in \S\ref{sec:wt1:ppp}. In \S\ref{sec:wt1:exp} we describe an 
approach that, although not powerful enough to prove our main theorem, can rule out non-zero holomorphic Jacobi forms of weight one  for infinitely many indexes and levels.
Our presentation (and in particular, notation) is similar to that which appears in \S2 of \cite{omjt}.

\subsection{Weil Representations}\label{sec:wt1:weil}

Write ${\Mpt}(\ZZ)$ for the metaplectic double cover of $\SL_2(\ZZ)$. We may realize $\Mpt(\ZZ)$ as the set of pairs $(\gamma,\upsilon)$ where $\gamma=\left(\begin{smallmatrix}a&b\\c&d\end{smallmatrix}\right)\in \SL_2(\ZZ)$ and $\upsilon:\HH\to\CC$ is a holomorphic function satisfying $\upsilon(\tau)^2=c\tau+d$. Then the product in $\Mpt(\ZZ)$ is given by $(\gamma,\upsilon(\tau))(\gamma',\upsilon'(\tau))=(\gamma\gamma',\upsilon(\gamma'\tau)\upsilon'(\tau))$. 
For generators we may take $\widetilde{T}:=(T,1)$  and $\widetilde{S}:=(S,\tau^\frac12)$ where $T:=\left(\begin{smallmatrix}1&1\\0&1\end{smallmatrix}\right)$ and $S:=\left(\begin{smallmatrix}0&-1\\1&0\end{smallmatrix}\right)$.

According to \cite{MR0332663} for example, if $\gamma=\left(\begin{smallmatrix}a&b\\c&d\end{smallmatrix}\right)\in\SL_2(\ZZ)$ is such that $c=0\xmod 4$ and $d=1\xmod 4$ then $j(\gamma,\tau):=\theta_{1,0}(\gamma\tau,0)\theta_{1,0}(\tau,0)^{-1}$ satisfies $j(\gamma,\tau)^2=c\tau+d$. So we may consider the pairs $(\gamma,j(\gamma,\tau))$ where $\gamma$ belongs to the {\em principal congruence subgroup} $\Gamma(4m)$, being the kernel of the natural map $\SL_2(\ZZ)\to\SL_2(\ZZ/4m\ZZ)$. These constitute a normal subgroup of $\Mpt(\ZZ)$ which we denote $\Gamma(4m)^*$.
The corresponding quotient $\Mpt(\ZZ)/\Gamma(4m)^*$ is a double cover---let us denote it $\widetilde{\SL_2}(\ZZ/4m\ZZ)$---of $\SL_2(\ZZ/4m\ZZ)$.

We will be concerned with the representations of $\Mpt(\ZZ)$ arising, via a construction of Weil, from (cyclic) finite quadratic spaces. A {\em finite quadratic space} is a finite abelian group $A$ equipped with a function $Q:A\to \QQ/\ZZ$ such that $Q(na)=n^2Q(a)$ for $n \in \ZZ$ and $a\in A$, and such that $B(a,b):=Q(a+b)-Q(a)-Q(b)$ is a $\ZZ$-bilinear map. 
Let $\CC A=\bigoplus_{a\in A} \CC e^a$ be the group algebra of $A$. 
The {\em Weil representation} associated to $(A,Q)$ is the left $\Mpt(\ZZ)$-module structure on $\CC A$ defined by requiring that
\begin{gather}\label{eqn:wt1:weil-weilrep}
\begin{split}
	\widetilde{T}e^a&=e(Q(a))e^a,\\
	\widetilde{S}e^a&=\sigma_A{|A|^{-\frac12}}\sum_{b\in A}e(-B(a,b))e^b,
\end{split}
\end{gather} 
for $a\in A$, where $\sigma_A:=|A|^{-\frac12}\sum_{a\in A}e(-Q(a))$. In general $\sigma_A$ is an eighth root of unity. The action (\ref{eqn:wt1:weil-weilrep}) factors through $\SL_2(\ZZ)$ if and only if $\sigma_A^4=1$.

The 
$\th_{m,r}$ (cf. (\ref{eqn:jac-thmr})) furnish concrete realizations of such Weil representations.
To explain this let 
$\Th_m$ denote the vector space spanned by the 
$\theta_{m,r}$ for $r\in \ZZ/2m\ZZ$, 
and define a 
right $\Mpt(\ZZ)$-module structure on $\Th_m$ by setting
\begin{gather}\label{eqn:wt1-Mp2action}
	(\f|_{\frac12,m}(\gamma,\upsilon))(\tau,z) :=  \f\left(\frac{a\tau+b}{c\tau+d},\frac{z}{c\tau+d}\right)\frac1{\upsilon(\tau)}e\left(-m\frac{cz^2}{c\tau+d}\right)
\end{gather}
for $\f\in \Th_m$ and $(\gamma,\upsilon)\in \Mpt(\ZZ)$, when $\gamma=\left(\begin{smallmatrix}a&b\\c&d\end{smallmatrix}\right)$. Then $\Th_m$ is isomorphic to the dual of the Weil representation attached to $D_m:=(\ZZ/2m\ZZ,x\mapsto\frac{x^2}{4m})$.
Explicitly, we may associate $\theta_{m,r}$ to the linear map $e^a\mapsto \delta_{a,r}$ on $\CC D_m$. (Note that $\sigma_A=e(-\frac18)$ in (\ref{eqn:wt1:weil-weilrep}) for $A=D_m$, for any $m$.)
Also, $\Th_m$ is trivial for the action of $\Gamma(4m)^*$, so may be regarded as a representation of $\Mpt(\ZZ/4m\ZZ)$.

Except when $m=1$ the $\Th_m$ are not irreducible. 
We will make use of the explicit decomposition into irreducible $\Mpt(\ZZ)$-modules which is given in Satz 1.8 of \cite{Sko_Thesis}. To describe this, note first that the orthogonal group 
$O_m:=\left\{a\in \ZZ/2m\ZZ\mid a^2=1\xmod 4m\right\}$
of $D_m$ acts naturally on $\Th_m$, via 
$\theta_{m,r}\cdot a:=\theta_{m,ra}$ for $r\in \ZZ/2m\ZZ$ and $a\in O_m$,
and this commutes with the action of $\Mpt(\ZZ)$. 
Next, define a Hermitian inner product on $\Th_m$ by requiring that the $\theta_{m,r}$ 
furnish an orthonormal basis. Then the action of $\Mpt(\ZZ)$ on $\Th_m$ is unitary with respect to this inner product, according to Lemma 1.4 of \cite{Sko_Thesis}. 
For $d$ an integer the assignment $U_d:\f(\tau,z)\mapsto \f(\tau,dz)$ defines an $\Mpt(\ZZ)$-module map $\Th_m\to \Th_{md^2}$. 
Define $\Th_m^\new$ to be the orthogonal complement in $\Th_m$ of the subspace spanned by the images of the maps $U_d$ for $d$ positive and $d^2|m$. 
The action of $O_m$ preserves $\Th_m^\new$ so for $\alpha\in \widehat{O}_m:=\hom(O_m,\CC^\times)$ we may define a sub $\Mpt(\ZZ)$-module of $\Th_m^\new$ by setting
\begin{gather}
	\Th_m^{\new,\alpha}:=\left\{\phi\in \Th_m^\new\mid \phi\cdot a = \alpha(a)\phi\right\}.
\end{gather}
Then 
$\Th_m^{\new,\alpha}$ is irreducible for $\Mpt(\ZZ)$, and the $\Mpt(\ZZ)$-irreducible submodules of $\Th_m$ are exactly the $\Th_{m/d^2}^{\new,\alpha}|U_d$ where 
$d^2|m$, 
according to Satz 1.8 of \cite{Sko_Thesis}. Note that the decomposition $\Th_m=\Th_m^-\oplus \Th_m^+$ is preserved by $\Mpt(\ZZ)$, where
$\Th_m^\pm$ denotes the span of the $\th_{m,r}^\pm$
(cf. (\ref{eqn:jac-thmrpm})). Indeed, the maps $U_d$ furnish an $\Mpt(\ZZ)$-module isomorphism
\begin{gather}\label{eqn:wt1-thmpmdec}
	\Th_m^\pm\simeq \bigoplus_{d^2|m}\bigoplus_{\substack{\alpha\in \widehat{O}_{m/d^2}\\ \alpha(-1)=\pm 1 }}\Th_{m/d^2}^{\new,\alpha}.
\end{gather}

\subsection{Exponents}\label{sec:wt1:exp}

In this short section we present a simple criterion (Lemma \ref{lem:wt1-exps}) which, although not powerful enough to handle all the cases of Theorem \ref{thm:wt1-main}, can be used to prove the vanishing of $J_{1,m}(N)$ for many cases in which $N$ is not square-free. 

To prepare for the proof note that since $\Gamma(4m)^*$ acts trivially on the $\theta_{m,r}$, 
the theta-coefficients $h_r$ (cf. (\ref{eqn:jac-thtdec})) 
of a holomorphic Jacobi form $\xi(\tau,z)$ of weight $1$, index $m$ and level $N$  
belong to the space $M_{\frac12}(\Gamma_0(N)\cap\Gamma(4m))$. 
Here
$M_{\frac12}(\Gamma)$, for $\Gamma$ a subgroup of some $\Gamma(4m)$, is the vector space of holomorphic functions $h:\HH\to \CC$ such that $h|_{\frac12}(\gamma,j(\gamma,\tau))=h$ for $\gamma\in \Gamma$, and $h|_{\frac12}(\gamma,\upsilon)$ remains bounded as $\Im(\tau)\to \infty$ for every $(\gamma,\upsilon)\in \Mpt(\ZZ)$, where
\begin{gather}
	(h|_{\frac12}(\gamma,\upsilon))(\tau):=h\left(\frac{a\tau+b}{c\tau+d}\right)\frac1{\upsilon(\tau)}.
\end{gather}
Since $\Gamma(4m)^*$ is normal in $\Mpt(\ZZ)$ the space $M_{\frac12}(\Gamma(4m))$ is naturally an $\Mpt(\ZZ)$-module. Also, $U_0:\f(\tau,z)\mapsto \f(\tau,0)$ defines a map $\Th_m\to M_{\frac12}(\Gamma(4m))$. 
\begin{prop}\label{prop:wt1-exps}
Suppose that $\xi\in J_{1,m}(N)$ for some positive integers $m$ and $N$. Choose $M$ so that $m|M$ and $N|4M$. 
Then the theta-coefficients of $\xi$ belong to 
$\sum_{M'|M}\Th_{M'}|U_0$.
\end{prop}
\begin{proof}
Assume that $m$, $N$ and $M$ are as in the statement of the proposition. Then the theta coefficients of $\xi$ belong to $M_\frac12(\Gamma(4M))$. According to Satz 5.2 of \cite{Sko_Thesis} we have 
\begin{gather}
M_{\frac12}(\Gamma(4M))=\bigoplus_{M'|M}\bigoplus_{\substack{\a'\in \widehat{O}_{M'}\\\a'(-1)=1}}\Th_{M'}^{\new,\a'}|U_0.
\end{gather}
The claimed result follows.
\end{proof}

Proposition \ref{prop:wt1-exps} 
puts restrictions on the exponents that can appear in the theta-coefficients of a candidate non-zero Jacobi form of weight $1$ with given index and level. On the other hand, the fact that the Fourier coefficients of such a Jacobi form can only involve integer powers of $q$ also imposes restrictions. Taking these together
we can rule out all but the zero function in many instances, especially when the index is small. 

\begin{lem}\label{lem:wt1-exps}
Suppose that $M$ is a positive integer and $m$ is a positive divisor of $M$. Set $m'=\frac{M}{m}$. If the equation 
\begin{gather}\label{eqn:wt1-exps}
r^2m'+s^2t=0\xmod 4M\end{gather}
has no solutions $(r,s,t)$ with $0<r<m$ and $t$ a divisor of $M$ then $J_{1,m}(4M)=\{0\}$.
\end{lem}
\begin{proof}
According to Proposition \ref{prop:wt1-exps}, if $\xi=\sum h_r\theta_{m,r}$ is the theta decomposition of a Jacobi form $\xi\in J_{1,m}(4M)$ then $h_r$ belongs to $\sum_{M'|M}\Th_{M'}|U_0$ for all $r$. Let $M'|M$ and set $t=\frac{M}{M'}$. Then $\theta_{M',s}(4M\tau,0)\in q^{s^2t}\ZZ[[q^{4M}]]$. On the other hand $h_r(4M\tau)\in q^{-r^2m'}\ZZ[[q^{4M}]]$. So in order for $h_r$ to be a non-zero linear combination of the $\theta_{M',s}(\tau,0)$ with $M'|M$ we require that (\ref{eqn:wt1-exps}) hold, for some $s$. That is, if $h_r$ is a non-zero element of $\sum_{M'|M}\Th_{M'}|U_0$ then we must have (\ref{eqn:wt1-exps}) for some $r$ and $s$, and some $t|M$. We have $h_r=-h_{-r}$ since $\xi$ has weight $1$, so we may assume that $0<r<m$. This proves the claim.
\end{proof}

The next result serves as an application of Lemma \ref{lem:wt1-exps}.

\begin{prop}\label{prop:wt1-expsapp}
The spaces $J_{1,2}(2^a)$, $J_{1,3}(4\cdot 3^a)$ and $J_{1,4}(2^a)$ are zero-dimensional for any non-negative integer $a$.
\end{prop}
\begin{proof}
We obtain the vanishing of $J_{1,2}(2^a)$ by applying Lemma \ref{lem:wt1-exps} with $m=2$ and $M$ an arbitrary positive power of $2$. Indeed, in this case $r=1$ so taking $M=2^a$ in (\ref{eqn:wt1-exps}) we obtain
$2^bs^2+2^{a-1}=0\xmod 2^{a+2}$
where $0\leq b\leq a$. If $s$ is a solution then $2^bs^2=2^{a-1}d$ for some $d=7\xmod 8$, but $7$ is neither even nor a square modulo $8$ so there is no such $s$.
So $J_{1,2}(2^{a+2})$ vanishes according to Lemma \ref{lem:wt1-exps}. The vanishing of $J_{1,3}(4\cdot 3^a)$ and $J_{1,4}(2^a)$ is obtained very similarly, by taking $m=3$ and $m=4$, respectively, and $M=m^a$. 
\end{proof}

\subsection{Prime Power Parts}\label{sec:wt1:ppp}

The module $\Th_m^{\new,\a}$ factors through 
$\Mpt(\ZZ/4m\ZZ)$ so it is natural to consider its prime power parts. To explain what this means
note that if $m_p$ denotes the largest power of $p$ dividing $m$ then the natural map
\begin{gather}\label{eqn:wt1:pmr-Mpt4mpparts}
	\Mpt(\ZZ/4m\ZZ)\to \Mpt(\ZZ/4m_2\ZZ)\times \prod_{p\text{ odd prime}}\SL_2(\ZZ/m_p\ZZ)
\end{gather}
is an isomorphism. Thus any irreducible $\Mpt(\ZZ/4m\ZZ)$-module can we written as an external tensor product of irreducible modules for $\Mpt(\ZZ/4m_2\ZZ)$ and the $\SL_2(\ZZ/m_p\ZZ)$. At the level of characters, if $\chi$ is an irreducible character for $\Mpt(\ZZ/4m\ZZ)$ then there are corresponding characters $\chi_p$ of $\Mpt(\ZZ/4m_2\ZZ)$ and the $\SL_2(\ZZ/m_p\ZZ)$ such that 
\begin{gather}\label{eqn:wt1:pmr-chipparts}
\chi(\gamma)=\prod_{p|4m}\chi_p(\gamma_p)
\end{gather} 
for $\gamma\in \Mpt(\ZZ/4m\ZZ)$, where $\gamma_p$ is the $\Mpt(\ZZ/4m_2\ZZ)$ or $\SL_2(\ZZ/m_p\ZZ)$ component of the image of $\gamma$ under the map (\ref{eqn:wt1:pmr-Mpt4mpparts}). Call $\chi_p$ the {\em $p$-part} of $\chi$. 

Even though 
$\CC D_m$---being the left $\Mpt(\ZZ)$-module dual to the right $\Mpt(\ZZ)$-module structure on $\Th_m$---is generally not irreducible, it has well-defined $p$-parts. Specifically, 
if we set $D_{m}{(a)}:=\left(\ZZ/2m\ZZ,\frac{ax^2}{4m}\right)$
for $a$ coprime to $m$, and 
$L_{m}{(a)}:=\left(\ZZ/m\ZZ,\frac{ax^2}{m}\right)$ for $a$ coprime to $m$ and $m$ odd, then 
we have 
\begin{gather}\label{eqn:wt1:pmr-CDmpparts}
\CC D_m \simeq \CC D_{m_2}(a_2)\otimes\bigotimes_{p\text{ odd prime}} \CC L_{m_p}(a_p) 
\end{gather}
for certain $a_p\in \ZZ$, where 
the factors in (\ref{eqn:wt1:pmr-CDmpparts})  correspond to the factors in 
(\ref{eqn:wt1:pmr-Mpt4mpparts}).
A calculation reveals that 
$a_2$ is an inverse to $\frac{m}{m_2}$ modulo $4m_2$, and $a_p$ an inverse to $\frac{4m}{m_p}$ modulo $m_p$.

Write $\CC D_m(a)^\pm$ for the submodule of $\CC D_m(a)$ spanned by the $e^a\pm e^{-a}$ for $a\in \ZZ/2m\ZZ$.
For $m$ odd write $\CC L_m(a)^\pm$ for the submodule of $\CC L_m(a)$ spanned by the $e^a\pm e^{-a}$ for $a\in \ZZ/m\ZZ$. If $m=p^k$ for $p$ an odd prime and $k\geq 2$ then $\CC L_{p^k}(a)^\pm$ admits a natural embedding by $\CC L_{p^{k-2}}(a)^\pm$. Let $\CC L_{p^k}(a)^{\new,\pm}$ denote the orthogonal complement of the submodule $\CC L_{p^{k-2}}(a)^\pm$ in $\CC L_{p^k}(a)^\pm$. 
Set $\CC L_p(a)^{\new,\pm}:=\CC L_p(a)^\pm$.
\begin{lem}[\!\!\cite{MR2512363}]
For $p$ an odd prime, $k$ a positive integer and $a$ coprime to $p^k$, the $\Mpt(\ZZ)$-modules $\CC L_{p^k}(a)^{\new,\pm}$ are irreducible.
\end{lem}

Let $\Th_m^\pm$ be the submodule of $\Th_m$ spanned by the $\th_{m,r}^\pm$ (cf. \ref{eqn:jac-thmrpm}). Set $L_{p^k}^{\new,\pm}:=L_{p^k}(1)^{\new,\pm}$ for $p$ an odd prime. We will employ the following notation for the characters of the $\Mpt(\ZZ)$-modules we have defined.
\begin{gather}
	\begin{split}\label{eqn:wt1:weil-chrs}
	\vartheta_m(\gamma)&:= \tr(\gamma| \Th_m)\\
	\vartheta_m^\pm(\gamma)&:=\tr(\gamma|\Th_m^\pm)\\
	\nu_m^\a(\gamma)&:=\tr(\gamma| \Th_m^{\new,\a})\\
	\lambda_{p^k}^\pm(\gamma)&:=\tr(\gamma|L_{p^k}^{\new,\pm})
	\end{split}
\end{gather}

To describe the $p$-parts of the $\nu_m^\a$ it is convenient to abuse notation and evaluate $\a$ on primes. 
Specifically, if $m_p$ denotes the highest power of $p$ dividing $m$, then there is a unique $a\xmod 2m$ such that $a=-1\xmod 2m_p$ and $a=1\xmod \frac{2m}{m_p}$, and we write $\a(p)$ instead of $\a(a)$. With this convention\footnote{Note the difference with \cite{omjt}, wherein characters $\a\in \hom(O_m,\CC^\times)$ are evaluated on exact divisors of $m$.} it follows from (\ref{eqn:wt1:pmr-CDmpparts}) that the $2$-part of $\nu_m^\a$ takes the form $\sigma(\vartheta_{m_2}^{\pm})$ where $\a(2)=\pm 1$ and $\sigma$ is the Galois automorphism of $\QQ(e(\tfrac{1}{4m_2}))$ mapping $e(\tfrac{1}{4m_2})$ to $e(\tfrac{a_2}{4m_2})$, for 
$a_2$ 
as in (\ref{eqn:wt1:pmr-CDmpparts}).
For $p$ odd the $p$-part of $\nu_m^\a$ is $\sigma(\lambda_{m_p}^\pm)$ where $\a(p)=\pm 1$, and now $\sigma$ maps $e(\tfrac{1}{m_p})$ to $e(\tfrac{a_p}{m_p})$ where 
$a_p$ is as in (\ref{eqn:wt1:pmr-CDmpparts}).

\subsection{Proof of the Main Result}\label{sec:wt1:pmr}

We present the proof of Theorem \ref{thm:wt1-main} in this section.
Our main technical tool is Lemma \ref{lem:wt1-mnot0mod8nnot0mod32mdiv1mod4ndivnocb}, which we establish by applying a specialization (Proposition \ref{prop:wt1-precise}) of Theorem 8 in \cite{MR2512363}.

Note that if $m$ and $m'$ are both positive divisors of some integer $M$ then the product $\vartheta_m\vartheta_{m'}$ (cf. (\ref{eqn:wt1:weil-chrs})) factors through 
$\SL_2(\ZZ/4M\ZZ)$. 
It will be useful to know when an irreducible constituent of $\vartheta_m\vartheta_{m'}$ can factor through $\SL_2(\ZZ/4M\ZZ)/\{\pm I\}$. Our first result in this section is an answer to this question. 
\begin{lem}\label{lem:wt1:pmr-Sevs}
Let $m$ and $m'$ be positive integers, and suppose that $\sigma$ and $\sigma'$ are Galois automorphisms of $\QQ(e(\tfrac{1}{4m}))$ and $\QQ(e(\tfrac{1}{4m'}))$ respectively. If $\sigma(i)=\sigma'(i)$ then 
the character $\sigma(\vartheta_m^-)\sigma'(\vartheta_{m'}^+)$ factors through $\SL_2(\ZZ)/\{\pm I\}$, but $-I$ acts non-trivially on every irreducible constituent of $\sigma(\vartheta_m^-)\sigma'(\vartheta_{m'}^-)$ and $\sigma(\vartheta_m^+)\sigma'(\vartheta_{m'}^+)$. If $\sigma(i)=-\sigma'(i)$ then 
the characters $\sigma(\vartheta_m^-)\sigma'(\vartheta_{m'}^-)$ and $\sigma(\vartheta_m^+)\sigma'(\vartheta_{m'}^+)$ factor through $\SL_2(\ZZ)/\{\pm I\}$, but $-I$ acts non-trivially on every irreducible constituent of $\sigma(\vartheta_m^-)\sigma'(\vartheta_{m'}^+)$. 
\end{lem}
\begin{proof}
The dimension of the $i^a$ eigenspace for the action of $S$ on $\Th_m^\pm\otimes \Th_{m'}^\pm$ (for any choice of signs) is 
\begin{gather}\label{eqn:wt1:weil-Ses}
\frac14\sum_{k=0}^{3} (-i)^{ak}
\vartheta_m^\pm(\widetilde{S}^k)\vartheta_{m'}^\pm(\widetilde{S}^k).
\end{gather}
Set $\zeta:=e(\tfrac18)$. From the definition (\ref{eqn:wt1:weil-weilrep}) we compute using identities for quadratic Gauss sums that 
\begin{gather}\label{eqn:wt1:weil-Straces}
\vartheta_m^\pm(\widetilde{S}^k)=
\begin{cases}
	(m\pm 1)(-i)^k\zeta^{\pm k} 
	&\text{ if $k=0\xmod 2$,}\\
	(-i)^k\zeta^{\pm k} 
	&\text{ if $k=1\xmod 2$ and $m=0\xmod 2$,}\\
	0&\text{ if $k=1\xmod 2$ and $m=1\xmod 2$.}
\end{cases}
\end{gather}
Applying (\ref{eqn:wt1:weil-Straces}) to (\ref{eqn:wt1:weil-Ses}) we find that the $\pm i$ eigenspaces for the action of $S$ on $\Th_m^-\otimes \Th_{m'}^+$ are trivial, as are the $\pm 1$ eigenspaces for the actions of $S$ on $\Th_m^-\otimes \Th_{m'}^-$ and $\Th_m^+\otimes \Th_{m'}^+$. This proves the claimed result for $\sigma$ and $\sigma'$ both trivial. The general case is very similar.
\end{proof}

The next result is a slight refinement of a specialization of Theorem 8 in \cite{MR2512363}. 
\begin{prop}[\!\!\cite{MR2512363}]\label{prop:wt1-precise}
Let $m$ and $N$ be positive integers, and let 
$M$ be a 
positive integer such that $m|M$ and $N$ divides $4M$. 
Then we have
\begin{gather}\label{eqn:wt1-precise}
	\dim J_{1,m}(N)
	=
	\sum_{\substack{m'|M\\ 
	\frac{M}{m'}\text{ square-free}}}
	\langle \vartheta_{m}^{-}\vartheta^+_{m'},\hat1_N \rangle
\end{gather}
where $\hat1_N$ is the character of $\SL_2(\ZZ)$ obtained via induction from the trivial character of $\Gamma_0(N)$.
\end{prop}
The hypothesis on $M$ ensures that $\hat1_N$ in (\ref{eqn:wt1-precise}) factors through $\SL_2(\ZZ/4M\ZZ)$. So we may regard the inner product $\langle\cdot\,,\cdot\rangle$ in (\ref{eqn:wt1-precise}) as the usual scalar product on class functions for $\SL_2(\ZZ/4M\ZZ)$.
\begin{proof}[Proof of Proposition \ref{prop:wt1-precise}]
In the notation of \cite{MR2512363} we take $F$ to be the $1\times 1$ matrix $(m)$, we set $\Gamma=\Gamma_0(N)$, and let $V$ be the trivial representation of $\Gamma$.
Applying Theorem 8 with these choices  we obtain that $J_{1,m}(N)$ is isomorphic to the space of $\Mpt(\ZZ)$-invariant vectors in
\begin{gather}\label{eqn:wt1-precise-proof}
	\bigoplus_{\substack{m'|M\\\frac{M}{m'} \text{ square-free}}}	\CC D_{m'}(-1)^+\otimes \CC D_m(-1)\otimes \hat V
\end{gather}
(the $l$ and $m$ in loc. cit. are $m'$ and $M$ here, respectively), where $\CC D_{m'}(-1)^+$ denotes the subspace of $\CC D_{m'}(-1)$ spanned by vectors $e^a+e^{-a}$ for $a\in \ZZ/2m'\ZZ$, and $\hat V$ is the $\SL_2(\ZZ)$-module obtained via induction from $V$. By our hypotheses, all the modules in (\ref{eqn:wt1-precise-proof}) factor through $\Mpt(\ZZ/4M\ZZ)$. For this group, the character of $\CC D_{m'}(-1)^+$ is $\overline{\vartheta_{m'}^+}$, and the character of $\CC D_m(-1)$ is $\overline{\vartheta_m}$. So if we write $\hat1_N$ for the character of $\hat V$, and write $1$ for the trivial character of $\Mpt(\ZZ/4M\ZZ)$, then the space of $\Mpt(\ZZ)$-invariants in the summand $\CC D_{m'}(-1)^+\otimes \CC D_m(-1)\otimes \hat V$ is the scalar product 
$\langle 1, \overline{\vartheta_{m'}^+\vartheta_{m}^{}}\hat1_N\rangle=\langle \vartheta_{m}^{}\vartheta_{m'}^+,\hat 1_N\rangle$
in the ring of class functions on $SL_2(\ZZ/4M\ZZ)$. 
Now $\hat 1_N$ factors through $\SL_2(\ZZ/4M\ZZ)/\{\pm I\}$ by construction, whereas $-I$ is non-trivial on every irreducible constituent of $\vartheta_m^+\vartheta_{m'}^+$ by Lemma \ref{lem:wt1:pmr-Sevs}. So $\langle \vartheta_{m}^{+}\vartheta_{m'}^+,\hat 1_N\rangle=0$ and $\langle \vartheta_{m}^{}\vartheta_{m'}^+,\hat 1_N\rangle=\langle \vartheta_{m}^{-}\vartheta_{m'}^+,\hat 1_N\rangle$. The claimed formula follows.
\end{proof}

Before presenting our main 
application of Proposition \ref{prop:wt1-precise}
we require some results about characters of $\Mpt(\ZZ/64\ZZ)$, and $\SL_2(\ZZ/p^2\ZZ)$ for $p$ an odd prime. The next lemma was obtained by working directly with the character table of $\Mpt(\ZZ/64\ZZ)$, which we constructed on a computer using GAP \cite{GAP4}. 
\begin{lem}\label{lem:wt1:pmr-sl264}
Suppose that $k,k'\in\{1,2,4,8,16\}$, and $\sigma$ and $\sigma'$ are Galois automorphisms of $\QQ(e(\tfrac{1}{64}))$ such that $\sigma(i)=\sigma(i')$. Then $\langle \sigma(\vartheta_k^-)\sigma'(\vartheta_{k'}^+),\hat1_{16}\rangle=0$. If $(k,k')$ is not $(8,2)$ or $(16,16)$ then $\langle \sigma(\vartheta_k^-)\sigma'(\vartheta_{k'}^+),\hat1_{32}\rangle=0$.
\end{lem}

\begin{lem}\label{lem:wt1:pmr-sl2p2}
If $p$ is an odd prime then $\hat1_{p^2}$ has four irreducible constituents. One of these is the trivial character, another has degree $p$, and the remaining two have degree $\frac12(p^2-1)$.
\end{lem}
\begin{proof}
Let $p$ be an odd prime. Set $G=\SL_2(\ZZ/p^2\ZZ)/\{\pm I\}$ and $\bar G=\SL_2(\ZZ/p\ZZ)/\{\pm I\}$, let $B$ be the image of $\Gamma_0(p^2)$ in $G$, 
and let $\bar B$ be the image of $\Gamma_0(p)$ in $\bar G$. Then $\hat1_{p^2}$ is the character of the permutation module $\CC[G/B]:=\CC G\otimes_{\CC B}\CC$, where $\CC$ on the right is the trivial module for $B$, and $\hat1_p$ is the character of $\CC[\bar{G}/\bar{B}]$. Let $W$ be the kernel of the natural map $\CC[G/B]\to \CC[\bar{G}/\bar{B}]$. 
Then $\hat1_{p^2}-\hat1_p$ is the character of $W$. The space $\CC[\bar{G}/\bar{B}]$ realizes the permutation module for the action of $\PSL_2(\ZZ/p\ZZ)$ on the points of the projective line over $\ZZ/p\ZZ$, and is well known to have two irreducible constituents, one being trivial. 
So $\hat1_p$ is the sum of the trivial character and an irreducible character of degree $p$. 

So we require to find the degrees of the irreducible constituents of $W$. 
We may construct these constituents explicitly. To do this define cosets $u_x:=\left(\begin{smallmatrix} x&-1\\1&0\end{smallmatrix}\right)B$ and $v_y:=\left(\begin{smallmatrix} 1&0\\y&1\end{smallmatrix}\right)B$ for $x,y\in \ZZ/p^2\ZZ$. Then the set $\{u_x,v_y\mid y=0\xmod p\}$ is a basis for $\CC[G/B]$. For $x\in \ZZ/p^2\ZZ$ and $a\in \ZZ/p\ZZ$ define $1$-dimensional vector spaces $U_x^a$ and $U_\infty^a$ in $\CC[G/B]$ by setting
\begin{gather}
U_x^a
:=\Span\left\{\sum_{k\xmod p} e(-\tfrac{ak}{p})u_{x+kp}\right\},\quad
U_\infty^a
:=\Span\left\{\sum_{k\xmod p} e(-\tfrac{ak}{p})v_{-kp}\right\}.
\end{gather}
Then $U_x^a$ only depends on $x\xmod p$, and $W$ is the direct sum of the $U_x^a$ with $a\neq 0$. With $S=\left(\begin{smallmatrix}0&-1\\1&0\end{smallmatrix}\right)$ and $T=\left(\begin{smallmatrix}1&1\\0&1\end{smallmatrix}\right)$ we have $T(U_x^a)=U_{x+1}^a$ and $S(U_x^a)=U_{-x^{-1}}^{a*x^2}$, where 
$a*x^2$ means $a$ in case $x\in \{0,\infty\}$, and is the usual product $ax^2$ otherwise. So if $W'$ is the sum of the $U_x^{a}$ where $a$ is restricted to non-zero quadratic residues, and if $W''$ is the sum of the $U_x^{a}$ where $a$ runs over the non-zero non-quadratic residues, then $W=W'\oplus W''$ as $G$-modules. 
We may observe directly that $W'$ and $W''$ are irreducible, or apply Frobenius reciprocity to show that $\#(B\backslash G/B)=4$ is an upper bound on the number of constituents of $\hat1_{p^2}$. We have $\dim W'=\dim W''=\frac12(p^2-1)$ so the lemma has been proved.
\end{proof}

The next result is our main application of Proposition \ref{prop:wt1-precise}, and our main technical tool for proving Theorem \ref{thm:wt1-main}.

\begin{lem}\label{lem:wt1-mnot0mod8nnot0mod32mdiv1mod4ndivnocb}
Let $m$ and $N$ be positive integers. 
Assume that none of $m$, $N$ or $mN$ are divisible by the cube of a prime that is congruent to $3$ modulo $4$.
Assume also that either $m$ is odd and $N$ is not divisible by $64$, 
or $m$ is congruent to $4$ modulo $8$ and $N$ is not divisible by $64$, 
or both $m$ and $N$ are not divisible by $32$. 
Then $J_{1,m}(N)$ is zero-dimensional.
\end{lem}
\begin{proof}
Since $J_{1,m}(N')$ is a subspace of $J_{1,m}(N)$ when $N'$ divides $N$ we may assume without loss of generality that 
$N$ is divisible by $16$. 
Let $P$ be the product of the prime divisors of $(m,N)$ that are congruent to $3$ modulo $4$. Then taking $M=\frac1{4P}mN$ in Proposition \ref{prop:wt1-precise} we see that the dimension of $J_{1,m}(N)$ is a sum of scalar products 
\begin{gather}\label{eqn:wt1-mnot0mod8nnot0mod32}
\langle \nu_{l}^{\a}\nu_{l'}^{\a'},\hat1_N\rangle
\end{gather}
(cf. (\ref{eqn:wt1:weil-chrs})), to be computed in the space of class functions on $\SL_2(\ZZ/4M\ZZ)$, where $l$ is a divisor of $m$, and $l'$ divides $\frac1{4P}mN$. 

Suppose that $p=3\xmod 4$ is a prime dividing $mN$. Then 
$p$ may occur as a divisor of $l$ or $l'$, or possibly both, in a summand (\ref{eqn:wt1-mnot0mod8nnot0mod32}). We claim that such a summand (\ref{eqn:wt1-mnot0mod8nnot0mod32}) must vanish unless $l$ and $l'$ are both exactly divisible by $p$, or both exactly divisible by even powers of $p$. Write $l_p$ for the highest power of $p$ dividing $l$, and interpret $l'_p$ similarly. 
By our conditions on $m$ and $N$ the remaining possibilities are that $(l_p,l'_p)$ is $(p^2,p)$, $(p,p^2)$, $(p,1)$ or $(1,p)$. Also, in the former two cases $N$ must be coprime to $p$. 
If $(l_p,l_p')=(p^2,p)$ then the $p$-part of $\nu_l^\a\nu_{l'}^{\a'}$ is $\sigma(\lambda_{p^2}^\pm)\sigma'(\lambda_p^\pm)$ (cf. (\ref{eqn:wt1:weil-chrs})) for some Galois automorphisms $\sigma,\sigma'$ of $\QQ(e(\tfrac{1}{p^2}))$, and the $p$-part $(\hat1_N)_p$ of $\hat1_N$ is trivial. But $\sigma(\lambda_{p^2}^\pm)$ and $\sigma'(\lambda_p^\pm)$ are irreducible of different degrees, so the trivial character cannot arise as a constituent of their product. The case that $(l_p,l_p')=(p,p^2)$ is similar. For $(l_p,l_p')$ equal to $(p,1)$ or $(1,p)$ the $p$-part of 
$\nu_l^{\a}\nu_{l'}^{\a'}$ is $\sigma(\lambda_p^\pm)$ for some Galois automorphism $\sigma$ of $\QQ(e(\tfrac{1}{p}))$. However, 
$(\hat1_N)_p$ 
is now the permutation character of $\SL_2(\ZZ)$ arising from its action on cosets of $\Gamma_0(N_p)$ where $N_p$ is the largest power of $p$ dividing $N$ (namely, $1$, $p$ or $p^2$). 
As such, 
the irreducible constituents of $(\hat1_N)_p=\hat1_{N_p}$ have degrees $1$ (the trivial character), $p$, or $\frac12(p^2-1)$, according to Lemma \ref{lem:wt1:pmr-sl2p2}.
Certainly $\frac{1}{2}(p+1)$ is not $1$ or $p$ or $\frac12(p^2-1)$ for any odd prime $p$, and $\frac12(p-1)$ is not one of these degrees when $p>3$. 
Although $\CC L_3(a)^-$ is $1$-dimensional it is not trivial. So we may conclude that the summand (\ref{eqn:wt1-mnot0mod8nnot0mod32}) vanishes unless any prime that is $3$ modulo $4$ either divides both $l$ and $l'$ exactly, or divides both $l$ and $l'$ with even multiplicity.

So for the remainder of the proof we may restrict to summands (\ref{eqn:wt1-mnot0mod8nnot0mod32}) such that the odd parts of $l$ and $l'$ (i.e. $\frac{l}{l_2}$ and $\frac{l'}{l'_2}$) are congruent modulo $4$. 
For now, let us also restrict to the case that $m$ is odd and $N=32\xmod 64$. 
Then the $l$ in a summand (\ref{eqn:wt1-mnot0mod8nnot0mod32}) is 
odd, and $l'\neq 0\xmod 16$. In the case that $l'$ is odd we have $l=l'\xmod 4$, so 
the $2$-parts of $\nu_l^{\a}$ and $\nu_{l'}^{\a'}$ are both $\sigma(\nu_1^+)=\sigma(\vartheta_1^+)$, where $\sigma$ is the Galois automorphism of $\QQ(e(\frac14))$ that maps $e(\frac14)$ to $e(\frac{l}4)$. 
By Lemma \ref{lem:wt1:pmr-Sevs} no irreducible constituent of $\sigma(\vartheta_1^+)^2$ factors through $\SL_2(\ZZ)/\{\pm I\}$, 
so (\ref{eqn:wt1-mnot0mod8nnot0mod32}) vanishes in this case. If $l'=2\xmod 4$ then 
the $2$-part of $\nu_{l'}^{\a'}$ is $\sigma'(\vartheta_2^{\pm})$ (we have $\nu_2^\pm=\vartheta_2^\pm$) where $\sigma'$ is the Galois automorphism of $\QQ(e(\frac{1}{8}))$ such that $\sigma(e(\frac18))=e(\frac{a'}8)$ for $a'=\frac{l'}2 \xmod 8$, 
and the sign in $\vartheta_2^\pm$ is determined by $\a'$. We claim that the products $\sigma(\vartheta_1^+)\sigma'(\vartheta_2^\pm)$ have no constituents in common with $(\hat1_N)_2=\hat1_{32}$. Since $l=\frac{l'}2\xmod 4$ this follows from Lemma \ref{lem:wt1:pmr-Sevs} for $\sigma(\vartheta_1^+)\sigma'(\vartheta_2^+)$, and follows from Lemma \ref{lem:wt1:pmr-sl264} for 
$\sigma(\vartheta_1^+)\sigma'(\vartheta_2^-)$.  
A directly similar argument holds for $l'=4\xmod 8$ and $l'=8\xmod 16$, wherein 
the $2$-part of $\nu_{l'}^{\a'}$ is $\sigma'(\nu_{k'}^{\pm})$ for $k'=4$ or $k'=8$, 
and $\sigma'$ is the Galois automorphism of $\QQ(e(\frac{1}{4k'}))$ mapping $e(\frac{1}{4k'})$  to $e(\frac{a'}{4k'})$ for $a'$ an inverse to $\frac{l'}{k'}$ modulo $4k'$. 
Just as above, now using $l=\frac{l'}{k'}\xmod 4$, 
Lemmas \ref{lem:wt1:pmr-Sevs} and \ref{lem:wt1:pmr-sl264} show that 
$\sigma(\vartheta_1^+)\sigma'(\vartheta_{k'}^\pm)$ has no constituents in common with $\hat1_{32}$. 
We conclude that $\dim J_{1,m}(N)=0$ when $m$ is odd and $N=32\xmod 64$. 

The argument for $m=4\xmod 8$ and $N=32\xmod 64$ is the same except that we also have to consider products $\sigma(\vartheta_4^\pm)\sigma'(\vartheta_{k'}^\pm)$ where $k'\in \{1,2,4,8\}$ and $\sigma$ is now a Galois automorphism of $\QQ(e(\tfrac{1}{16}))$ such that $\sigma(i)=\sigma'(i)$. Lemmas \ref{lem:wt1:pmr-Sevs} and \ref{lem:wt1:pmr-sl264} show that such products have no constituents in common with $\hat1_{32}$. Indeed, these same lemmas show that the $\sigma(\vartheta_k^\pm)\sigma'(\vartheta_{k'}^\pm)$ and $\sigma(\vartheta_k^\mp)\sigma'(\vartheta_{k'}^\pm)$ have no constituents in common with $\hat1_{16}$ when $k,k'\in\{1,2,4,8,16\}$, and $\sigma$ and $\sigma'$ are suitable Galois automorphisms satisfying $\sigma(i)=\sigma(i')$. This handles the case that $m$ and $N$ are both not divisible by $32$, and completes the proof of the claim. 
\end{proof}

Note that $\nu_{p^2}^+$ does occur as a constituent of $\hat1_N$ when $p$ is an odd prime and 
$p^2|N$. This is the main reason for our restriction on odd primes congruent to $3\xmod 4$ 
in the statement of Lemma \ref{lem:wt1-mnot0mod8nnot0mod32mdiv1mod4ndivnocb}. 
Lemma \ref{lem:wt1:pmr-sl264} explains our restrictions on powers of $2$.

Lemma \ref{lem:wt1-mnot0mod8nnot0mod32mdiv1mod4ndivnocb} can handle all but a few of the cases of Theorem \ref{thm:wt1-main}. For the remainder we require to show the vanishing of $J_{1,m}(N)$ in some instances where $mN$ is divisible by $27$. We obtain this via three more specialized applications of Proposition \ref{prop:wt1-precise}, which we now present. Our proofs will apply properties of the character tables of $\SL_2(\ZZ/9\ZZ)$ and $\SL_2(\ZZ/25\ZZ)$ which we verified using GAP \cite{GAP4}.

\begin{lem}\label{lem:wt1-m3n144}
The space $J_{1,3}(144)$ is zero-dimensional.
\end{lem}
\begin{proof}
Applying Proposition \ref{prop:wt1-precise} with $m=3$, $M=36$ and $N=144$ we see that the dimension of $J_{1,3}(144)$ is a sum of terms
\begin{gather}\label{eqn:wt1-m3n144}
\langle \nu_{3}^-\nu_{l'}^{\a'},\hat1_{144}\rangle
\end{gather}
(cf. (\ref{eqn:wt1:weil-chrs})) where $l'$ is a divisor of $36$ and $\a'(2)\a'(3)=1$. 
The $3$-part of $\hat1_{144}$ is $(\hat1_{144})_3=\hat1_9$ 
and has no non-trivial constituents with degree less than $2$ (cf. Lemma \ref{lem:wt1:pmr-sl2p2}). But the $3$-part of $\nu_{3}^-$ is $\lambda_3^-$, which is a non-trivial character of degree $1$. So if $3$ does not divide $l'$ then the $3$-part of $\nu_{l'}^{\a'}$ is trivial and (\ref{eqn:wt1-m3n144}) has a factor $\langle \lambda_3^-,\hat1_9\rangle$ which vanishes. So we may assume that $3$ divides $l'$. Suppose that $9$ does not divide $l'$. Then taking $k'$ to be the highest power of $2$ dividing $l'$ we have that $k'\in\{1,2,4\}$ and the $2$-part of $\nu_{3}^-\nu_{l'}^{\a'}$ is $\sigma(\vartheta_1^+)\sigma'(\vartheta_{k'}^\pm)$ where $\sigma$ and $\sigma'$ satisfy $\sigma(i)=\sigma'(i)=-i$. No such products have a constituent in common with $(\hat1_{144})_2=\hat1_{16}$ by Lemmas \ref{lem:wt1:pmr-Sevs} and \ref{lem:wt1:pmr-sl264}, so we may assume that $l'$ is divisible by $9$.
Then 
the $3$-part of $\nu_3^-\nu_{l'}^{\a'}$ is $\lambda_3^-\sigma'(\lambda_9^\pm)$ where the sign is such that $\a'(3)=\pm 1$, and $\sigma'(e(\tfrac19))=e(\frac{a'}{9})$ for 
$a'$ an inverse to $4k'$ modulo $9$.
The products $\lambda_3^-\sigma'(\lambda_9^-)$ have no constituents in common with $\hat1_9$ because $S$ only has $\pm i$ for eigenvalues on corresponding representations. By direct computation we find that $\langle \lambda_3^-\sigma'(\lambda_9^+),\hat1_9\rangle=0$ 
for all choices of $\sigma'$. 
We conclude that $J_{1,3}(144)$ vanishes, as required.
\end{proof}

\begin{lem}\label{lem:wt1-m6n36}
The space $J_{1,6}(36)$ is zero-dimensional.
\end{lem}
\begin{proof}
The proof is similar to that of Lemma \ref{lem:wt1-m3n144}, but a little more involved.
We apply Proposition \ref{prop:wt1-precise} with $m=6$, $N=36$ and $M=18$ to see that the dimension of $J_{1,6}(36)$ is a sum of terms
\begin{gather}\label{eqn:wt1-m6n36}
\langle \nu_{6}^\a\nu_{l'}^{\a'},\hat1_{36}\rangle
\end{gather}
where $\a(2)\a(3)=-1$, $l'$ is a divisor of $18$ and $\a'(2)\a'(3)=1$. 
The $3$-part of $\hat1_{36}$ is $(\hat1_{36})_3=\hat1_9$ 
and has no non-trivial constituents with degree less than $2$, but the $3$-part of $\nu_{6}^\a$ takes the form $\sigma(\lambda_3^\pm)$ and is thus a non-trivial character of degree $1$ or $2$. So arguing as for Lemma \ref{lem:wt1-m3n144} we restrict to the case that $3$ divides $l'$. If $l'$ is not divisible by $9$ then the $2$-part of $\nu_6^\a\nu_{l'}^{\a'}$ takes the form $\sigma(\vartheta_2^\pm)\sigma'(\vartheta_{k'}^\pm)$ for $k'\in\{1,2\}$ where $\sigma(i)=\sigma'(i)=-i$. No such products have constituents in common with $\hat1_4$ according to Lemmas \ref{lem:wt1:pmr-Sevs} and \ref{lem:wt1:pmr-sl264}, so we may assume that $l'$ is either $9$ or $18$. But if $l'=9$ then the $2$-part of $\nu_{6}^\a\nu_{l'}^{\a'}$ is $\sigma(\vartheta_2^\pm)\vartheta_{1}^+$ where $\sigma(i)=-i$. 
Lemma \ref{lem:wt1:pmr-Sevs} tells us that only $\sigma(\vartheta_2^+)\vartheta_1^+$ can have a constituent in common with $(\hat1_{36})_2=\hat1_4$, but it is irreducible of degree $6$, and $\hat1_4$ is degree $6$ but not irreducible. 

So we are left with the case that $l'=18$. Now the $2$-part of $\nu_{6}^\a\nu_{l'}^{\a'}$ is $\sigma(\vartheta_2^\pm)\vartheta_{2}^\pm$ where $\sigma(i)=-i$ as before. Applying Lemma \ref{lem:wt1:pmr-Sevs} we restrict to $\sigma(\vartheta_2^+)\vartheta_2^+$ and $\sigma(\vartheta_2^-)\vartheta_2^-$. The latter of these has degree $1$ but is not trivial, so is not a constituent of $\hat 1_4$, whereas the former does have a degree $2$ constituent in common with $\hat1_4$. So we have reduced to considering $\langle \nu_6^\a\nu_{18}^{\a'},\hat1_{36}\rangle$, where $\a(2)=1$, $\a(3)=-1$, $\a'(2)=1$ and $\a'(3)=1$. The $3$-part of this is $\langle \sigma(\lambda_3^-)\lambda_9^+,\hat1_9\rangle$, but this vanishes for both choices of $\sigma$. 
We conclude that $J_{1,6}(36)$ also vanishes.
\end{proof}

\begin{lem}\label{lem:wt1-m30n36}
The space $J_{1,30}(36)$ is zero-dimensional.
\end{lem}
\begin{proof}
Apply Proposition \ref{prop:wt1-precise} with $m=30$, $N=36$ and $M=90$ to see that the dimension of $J_{1,30}(36)$ is a sum of terms
\begin{gather}\label{eqn:wt1-m30n9}
\langle \nu_{30}^\a\nu_{l'}^{\a'},\hat1_{36}\rangle
\end{gather}
where $\a(2)\a(3)\a(6)=-1$, $l'$ is a divisor of $90$ and $\a'(2)\a'(3)\a'(5)=1$. 
The $5$-part of $\hat1_{36}$ is trivial, but the $5$-part of $\nu_{30}^\a$ is not, so we may assume that that $5$-part of $\nu_{l'}^{\a'}$ is also not trivial. That is, we may assume that $5$ divides $l'$. 
The $3$-part of $\nu_{30}^\a$ is a non-trivial character of degree $1$ or $2$ so we can restrict to the case that $3$ divides $l'$ just as we did for Lemma \ref{lem:wt1-m6n36}. Also as in that proof, consideration of $2$-parts shows that $l'$ must be divisible by $9$, so $l'$ is either $45$ or $90$, and further consideration of $2$-parts rules out $l'=45$.

So we may assume that $l'=90$. Now the $5$-part of $\nu_{30}^\a\nu_{90}^{\a'}$ is $\sigma(\lambda_5^\pm)\sigma'(\lambda_5^\pm)$ where $\sigma(e(\tfrac{1}{5}))=e(\tfrac{4}{5})$ and $\sigma'(e(\tfrac{1}{5}))=e(\tfrac{3}{5})$. Inspecting the characters we see that $\sigma(\lambda_5^\pm)$ is not dual to  $\sigma'(\lambda_5^\pm)$ for these particular $\sigma$ and $\sigma'$. We conclude that $J_{1,30}(36)$ vanishes, as we required to show.
\end{proof}

We are now prepared to present the proof of Theorem \ref{thm:wt1-main}.
\begin{proof}[Proof of Theorem \ref{thm:wt1-main}]
The proof is summarized in Appendix \ref{sec:levels}, where we display each of the conjugacy classes $[g]$ of each non-trivial group $G^X$ arising in umbral moonshine, together with levels $N_g$ such that the mock Jacobi form $\phi^X_g=\sum H^X_{g,r}\th_{m,r}$ corresponding to $[g]$ has the same multiplier system as a Jacobi form of weight $1$, index $m=m^X$ and level $N=N_g$. The validity of the given $N_g$ can be verified using the explicit descriptions of the $H^X_g$ in \cite{umrec}. Inspecting the tables in \S\ref{sec:levels} we find that Lemma \ref{lem:wt1-mnot0mod8nnot0mod32mdiv1mod4ndivnocb} implies the vanishing of $J_{1,m}(N)$ for all $[g]$ except for the cases shaded in 
{yellow} or 
{orange}. The 
{orange} cases are the two classes with $X=A_8^3$ for which theta series contributions to $H^X_g$ exist. The vanishing of the $J_{1,m}(N)$ corresponding to 
{yellow} classes is obtained by applying Lemma \ref{lem:wt1-m3n144} for $X=A_2^{12}$ and $[g]\in \{3B,6B,12A\}$, Lemma \ref{lem:wt1-m6n36} for $X=D_4^6$ and $[g]\in \{3C, 6C\}$, and Lemma \ref{lem:wt1-m30n36} for $X=E_8^3$ and $[g]=3A$. For $g=e$ in any $G^X$, the corresponding level is $1$, and the vanishing of $J_{1,m}(1)$ was proven much earlier in \cite{Sko_Thesis}.
The proof of the theorem is complete.
\end{proof}

A number of cases of Theorem \ref{thm:wt1-main} can also be established using Proposition \ref{prop:wt1-expsapp}, or more refined applications of Proposition \ref{prop:wt1-exps}.

 
\section*{Acknowledgements}

We thank Ken Ono for indicating to us that holomorphic Jacobi forms of weight one should exist, we thank Ralf Schmidt and Nils-Peter Skoruppa for answering our questions about their work, and we thank Cris Poor and David Yuen for correspondence on paramodular forms. We also thank Tomoyoshi Ibukiyama and Nils-Peter Skoruppa for sharing a draft \cite{IbuSko_Cor} of their correction to \cite{MR2379341}. Although we did not use their revised results (cf. Theorem \ref{thm:ibuskocor}) in this paper, we benefited from seeing the proofs. 

The work of M.C. is supported by ERC starting grant H2020 ERC StG 2014. J.D. acknowledges support from the U.S. National Science Foundation\footnote{Any opinions, findings, and conclusions or recommendations expressed in this material are those of the author(s) and do not necessarily reflect the views of the National Science Foundation.} (DMS 1203162, DMS 1601306), and from the Simons Foundation (\#316779). J.H. also acknowledges support from the U.S. National Science Foundation (PHY 1520748), and from the Simons Foundation (\#399639).

\clearpage

\appendix

\begin{table}
\section{Levels}\label{sec:levels}

\begin{center}
\caption{Conjugacy classes and levels at $\ll=2$, $\rs=A_1^{24}$}\label{tab:chars:eul:2}
\smallskip
\begin{tabular}{c|c c c c cccccccccc}
\toprule
$[g]$	&1A	&2A	&2B	&3A	&3B	&4A	&4B	&4C	&5A	&6A	&6B	\\
	\midrule
$n_g|h_g$&$1|1$&$2|1$&${2|2}$&$3|1$&$3|3$&$4|2$&$4|1$&${4|4}$&$5|1$&$6|1$&$6|6$\\
$N_g$&$1$&$2$&${4}$&$3$&$9$&$8$&$4$&${16}$&$5$&$6$&$36$\\
\bottomrule
\end{tabular}
\begin{tabular}{c|c c c c ccccccccccc}
\toprule
$[g]$	& 7AB	&8A	&10A	&11A&12A	&12B	&14AB	&15AB	&21AB	&23AB	\\
	\midrule
$n_g|h_g$&$7|1$&$8|1$&$10|2$&$11|1$&$12|2$&$12|12$&$14|1$&$15|1$&$21|3$&$23|1$\\
$N_g$&$7$&$8$&$20$&$11$&$24$&$144$&$14$&$15$&$63$&$23$\\
\bottomrule
\end{tabular}
\end{center}

\begin{center}
\caption{Conjugacy classes and levels at $\ll=3$, $\rs=A_2^{12}$}\label{tab:chars:eul:3}
\smallskip
\begin{tabular}{c|ccccccc>{\columncolor{yellow}}c>{\columncolor{yellow}}ccccc}
\toprule
$[g]$	&1A&   		2A&   		4A&   		2B&   		2C&   		3A&   		6A&   		3B&   		6B&   		4B& 	  		4C	\\
	\midrule
$n_g|h_g$&$1|1$&$1|4$&${2|8}$&$2|1$&$2|2$&$3|1$&$3|4$&${3|3}$&${3|12}$&$4|2$&$4|1$\\
$N_g$&$1$&$4$&$16$&$2$&$4$&$3$&$12$&$9$&$36$&$8$&$4$\\
\bottomrule
\end{tabular}
\begin{tabular}{c|cc>{\columncolor{yellow}}cccccccc}
\toprule
$[g]$	& 5A&   		10A&   		12A&   		6C&   		6D&   		8AB&   	 	8CD&   		20AB&   		11AB&   		22AB\\
	\midrule
$n_g|h_g$&$5|1$&$5|4$&$6|24$&$6|1$&$6|2$&$8|4$&$8|1$&${10|8}$&$11|1$&$11|4$\\
$N_g$&$5$&$20$&$144$&$6$&$12$&$32$&$8$&$80$&$11$&$44$\\
\bottomrule
\end{tabular}
\end{center}

\begin{center}
\caption{Conjugacy classes and levels at $\ll=4$, $\rs=A_3^8$}\label{tab:chars:eul:4}
\smallskip
\begin{tabular}{c|ccccccccccccc}\toprule
$[g]$&   		1A&   2A&   	2B&   	4A&			4B&			2C&   	3A&   	6A&   		6BC&   	8A&   	4C&   	7AB&   	14AB\\ 
	\midrule
$n_g|h_g$&$1|1$& $1|2$&	$2|2$&	$2|4$&			${4|{4}}$&	$2|1$& 	$3|1$& 	$3|2$&		$6|2$&	${4|{8}}$&		$4|1$&  	$7|1$&	$7|2$\\	
$N_g$&$1$& $2$&	$4$&	$8$&			$16$&	$2$& 	$3$& 	$6$&		$12$&	$32$&		$4$&  	$7$&	$14$
\\\bottomrule
\end{tabular}
\end{center}

\begin{center}
\caption{Conjugacy classes and levels at $\ll=5$, $\rs=A_4^6$}\label{tab:chars:eul:5}
\smallskip
\begin{tabular}{c|ccccccccccccc}\toprule
$[g]$&   		1A&		2A&   	2B&   	2C&			3A&			6A&   	5A&   	10A&   		4AB&   	4CD&	12AB\\ 
	\midrule
$n_g|h_g$&		$1|1$&	$1|4$&	$2|2$&	$2|1$&		$3|3$&		$3|12$&	$5|1$&	$5|4$&		$2|8$&		$4|1$&	$6|24$	\\	
$N_g$&		$1$&	$4$&	$4$&	$2$&		$9$&		$36$&	$5$&	$20$&		$16$&		$4$&	$144$	\\	
\bottomrule
\end{tabular}
\end{center}
\end{table}

\begin{table}
\begin{center}
\caption{Conjugacy classes and levels at $\ll=6$, $\rs=A_5^4D_4$}\label{tab:chars:eul:6}
\smallskip
\begin{tabular}{c|ccccccc}\toprule
$[g]$&   		1A&		2A&   	2B&   	4A&			3A&			6A&   	8AB\\ 
	\midrule
$n_g|h_g$&		$1|1$&	$1|2$&	$2|1$&	$2|2$&		$3|1$&		$3|2$&	$4|2$\\	
$N_g$&		$1$&	$2$&	$2$&	$4$&		$3$&		$6$&	$8$
	\\\bottomrule
\end{tabular}
\end{center}

\begin{center}
\caption{Conjugacy classes and levels at $\ll=6+3$, $\rs=D_4^6$}\label{tab:chars:eul:6+3}
\smallskip
\begin{tabular}{c|ccccc>{\columncolor{yellow}}ccccccccc>{\columncolor{yellow}}c}\toprule
$[g]$&   		1A&		3A&   	2A&   	6A&			3B&			3C&   	4A&	12A&	5A&	15AB&	2B&	2C&	4B&	6B&	6C\\ 
	\midrule
$n_g|h_g$&		$1|1$&	$1|3$&	$2|1$&	$2|3$&		$3|1$&		$3|3$&	$4|2$&	$4|6$&	$5|1$&	$5|3$&	$2|1$&	$2|2$&	$4|1$&	$6|1$&	$6|6$\\	
$N_g$&		$1$&	$3$&	$2$&	$6$&		$3$&		$9$&	$8$&	$24$&	$5$&	$15$&	$2$&	$4$&	$4$&	$6$&	$36$\\	
\bottomrule
\end{tabular}
\end{center}

\begin{center}
\caption{\label{tab:FrmG7fp}
Conjugacy classes and levels at $\ll=7$, $\rs=A_6^4$}\label{tab:chars:eul:7}
\smallskip
\begin{tabular}{c|ccccc}
\toprule
$[g]$&   1A&   2A&   4A&   3AB&   6AB\\ 
	\midrule
$n_g|h_g$&		$1|1$&	$1|4$&	$2|8$&	$3|1$&	$3|4$\\
$N_g$&		$1$&	$4$&	$16$&	$3$&	$12$
\\
\bottomrule
\end{tabular}
\end{center}

\begin{center}
\caption{Conjugacy classes and levels at $\ll=8$, $\rs=A_7^2D_5^2$}\label{tab:chars:eul:8}
\smallskip
\begin{tabular}{c|ccccc}\toprule
	$[g]$&	1A&	2A&	2B&2C&4A\\
		\midrule
$n_g|h_g$&		$1|1$&$1|2$&${2|1}$&$2|1$&${2|4}$\\	
$N_g$&		$1$&$2$&$2$&$2$&$8$
\\\bottomrule
\end{tabular}
\end{center}

\begin{center}
\caption{Conjugacy classes and levels at $\ll=9$, $\rs=A_8^3$}\label{tab:chars:ccl:9}
\smallskip
\begin{tabular}{c|cccc>{\columncolor{orange}}c>{\columncolor{orange}}c}\toprule
	$[g]$&	1A&	2A&	2B&2C&3A&6A\\
		\midrule
$n_g|h_g$&		$1|1$&$1|4$&${2|1}$&$2|2$&$3|3$&$3|12$\\	
$N_g$&		$1$&$4$&${2}$&$4$&$9$&$36$	\\\bottomrule
\end{tabular}
\end{center}

\begin{center}
\caption{Conjugacy classes and levels at $\ll=10$, $\rs=A_9^2D_6$}\label{tab:chars:eul:10}
\smallskip
\begin{tabular}{c|ccc}\toprule
	$[g]$&	1A&	2A&	4AB\\
		\midrule
$n_g|h_g$&		$1|1$&$1|2$&${2|2}$\\	
$N_g$&		$1$&$2$&$4$
	\\\bottomrule
\end{tabular}
\end{center}
\end{table}

\begin{table}
\begin{center}
\caption{Conjugacy classes and levels at $\ll=10+5$, $\rs=D_6^4$}\label{tab:chars:eul:10+5}
\smallskip
\begin{tabular}{c|ccccc}\toprule
	$[g]$&	1A&	2A&	3A&2B&4A\\
		\midrule
	$n_g|h_g$&$1|1$&$2|2$&$3|1$&$2|1$&$4|4$\\
	$N_g$&$1$&$4$&$3$&$2$&$16$
	\\\bottomrule
\end{tabular}
\end{center}
\begin{center}
\caption{Conjugacy classes and levels at $\ll=12$, $\rs=A_{11}D_7E_6$}\label{tab:chars:eul:12}
\smallskip
\begin{tabular}{c|cc}\toprule
	$[g]$&	1A&	2A\\
		\midrule
	$n_g|h_g$&		$1|1$&$1|2$\\	
	$N_g$&		$1$&$2$
\\\bottomrule
\end{tabular}
\end{center}
\begin{center}
\caption{Conjugacy classes and levels at $\ll=12+4$, $\rs=E_6^4$}\label{tab:chars:eul:12+4}
\smallskip
\begin{tabular}{c|ccccccc}\toprule
$[g]$&   		1A&		2A&   	2B&   	4A&			3A&			6A&   	8AB\\ 
	\midrule
$n_g|h_g$&	$1|1$&$1|2$&$2|1$&$2|4$&$3|1$&$3|2$&$4|8$\\
$n_g|h_g$&	$1$&$2$&$2$&$8$&$3$&$6$&$32$
	\\\bottomrule
\end{tabular}
\end{center}
\begin{center}
\caption{Conjugacy classes and levels at $\ll=13$, $\rs=A_{12}^2$}\label{tab:chars:eul:13}
\smallskip
\begin{tabular}{c|ccc}\toprule
	$[g]$&	1A&	2A&	4AB\\
		\midrule
$n_g|h_g$&		$1|1$&$1|4$&${2|8}$\\	
$N_g$&		$1$&$4$&${16}$
	\\\bottomrule
\end{tabular}
\end{center}
\begin{center}
\caption{Conjugacy classes and levels at $\ll=14+7$, $\rs=D_8^3$}\label{tab:chars:eul:14+7}
\smallskip
\begin{tabular}{c|ccc}\toprule
	$[g]$&	1A&	2A&	3A\\
		\midrule
	$n_g|h_g$&$1|1$&$2|1$&$3|3$\\
	$N_g$&$1$&$2$&$9$
	\\\bottomrule
\end{tabular}
\end{center}
\begin{center}
\caption{Conjugacy classes and levels at $\ll=16$, $\rs=A_{15}D_9$}\label{tab:chars:eul:16}
\smallskip
\begin{tabular}{c|cc}\toprule
	$[g]$&	1A&	2A\\
		\midrule
	$n_g|h_g$&$1|1$&$1|2$\\
	$N_g$&$1$&$2$
\\\bottomrule
\end{tabular}
\end{center}
\end{table}

\begin{table}

\begin{center}
\caption{Conjugacy classes and levels at $\ll=18$, $\rs=A_{17}E_7$}\label{tab:chars:eul:18}
\smallskip
\begin{tabular}{c|cc}\toprule
	$[g]$&	1A&	2A\\
		\midrule
	$n_g|h_g$&$1|1$&$1|2$\\
	$N_g$&$1$&$2$
\\\bottomrule
\end{tabular}
\end{center}
\begin{center}
\caption{Conjugacy classes and levels at $\ll=18+9$, $\rs=D_{10}E_7^2$}\label{tab:chars:eul:18+9}
\smallskip
\begin{tabular}{c|cc}\toprule
	$[g]$&	1A&	2A\\
		\midrule
	$n_g|h_g$&$1|1$&$2|1$\\
	$N_g$&$1$&$2$
\\\bottomrule
\end{tabular}
\end{center}
\begin{center}
\caption{Conjugacy classes and levels at $\ll=22+11$, $\rs=D_{12}^2$}\label{tab:chars:eul:22+11}
\smallskip
\begin{tabular}{c|cc}\toprule
	$[g]$&	1A&	2A\\
		\midrule
		$n_g|h_g$&$1|1$&$2|2$\\
		$N_g$&$1$&$2$
	\\\bottomrule
\end{tabular}
\end{center}
\begin{center}
\caption{Conjugacy classes and levels at $\ll=25$, $\rs=A_{24}$}\label{tab:chars:eul:25}
\smallskip
\begin{tabular}{c|cc}\toprule
	$[g]$&	1A&	2A\\
		\midrule
	$n_g|h_g$&	$1|1$&	$1|4$\\
	$N_g$&	$1$&	$4$
\\\bottomrule
\end{tabular}
\end{center}

\begin{center}
\caption{Conjugacy classes and levels at $\ll=30+6,10,15$, $\rs=E_8^3$}\label{tab:chars:eul:30+6,10,15}
\smallskip
\begin{tabular}{c|cc>{\columncolor{yellow}}c}\toprule
$[g]$&   		1A&		2A&   	3A\\ 
	\midrule
	$n_g|h_g$&	$1|1$&	$2|1$&	$3|3$\\
	$N_g$&	$1$&	$2$&	$9$
	\\\bottomrule
\end{tabular}
\end{center}
\end{table}

\clearpage

\section{Characters at $\ell=9$}\label{sec:chars}

\begin{table}[ht]
\begin{center}
\caption{Character table of ${G}^{\rs}\simeq \Dih_6$, $\rs=A_8^3$}\label{tab:chars:irr:9}
\smallskip
\begin{tabular}{c|c|rrrrrr}\toprule
$[g]$&FS&   1A&   2A&   2B&   2C&   3A&    6A\\
	\midrule
$[g^2]$&&	1A&	1A&	1A&	1A&	3A&	3A\\
$[g^3]$&&	1A&	2A&	2B&	2C&	1A&	2A\\
	\midrule
$\chi_1$&$+$&   $1$&   $1$&   $1$&   $1$&   $1$&   $1$\\
$\chi_2$&$+$&   $1$&   $1$&   $-1$&   $-1$&   $1$&   $1$\\
$\chi_3$&$+$&   $2$&   $2$&   $0$&   $0$&   $-1$&   $-1$\\
$\chi_4$&$+$&   $1$&   $-1$&   $-1$&   $1$&   $1$&   $-1$\\
$\chi_5$&$+$&   $1$&   $-1$&   $1$&   $-1$&   $1$&   $-1$\\
$\chi_6$&$+$&   $2$&   $-2$&   $0$&   $0$&   $-1$&   $1$\\\bottomrule
\end{tabular}
\end{center}
\end{table}

\begin{table}[ht]
\begin{center}
\caption{Twisted Euler characters and Frame shapes at $\ll=9$, $\rs=A_8^3$}\label{tab:chars:eul:9}
\begin{tabular}{l|rrrrrr}\toprule
	$[g]$&	1A&	2A&	2B&2C&3A&6A\\
		\midrule
$n_g|h_g$&		$1|1$&$1|4$&${2|1}$&$2|2$&$3|3$&$3|12$\\	
		\midrule
	$\bar{\chi}^{\rs_A}_{g}$	&3&3&1&1&0&0\\
	$\chi^{\rs_A}_{g}$		&3&-3&1&-1&0&0\\
		\midrule
	$\bar{\Pi}^{\rs_A}_{g}$	&$1^3$&$1^3$&$1^12^1$&$1^12^1$&$3^1$&$3^1$\\
	$\widetilde{\Pi}^{\rs_A}_{g}$&$1^{24}$&${2^{12}}$&$1^82^8$&$2^{12}$&$3^8$&${6^4}$
	\\\bottomrule
\end{tabular}
\end{center}
\end{table}

\clearpage

\begin{table}
\section{Coefficients at $\ell=9$}\label{sec:coeffs}
\vspace{-8pt}
\begin{minipage}[b]{0.49\linewidth}
\centering
\caption{\label{tab:coeffs:9_1}
$H^{(9)}_{g,1}$, $\rs=A_8^3$}\smallskip
\begin{tabular}{r|rrrrrr}\toprule
$[g]$	&1A	&2A	&2B	&2C	&3A	&6A	\\
\midrule
	$\G_g$&		$1$&$1|4$&$2$&$2|2$&	${3|3}$&	${3|12}$\\	
\midrule
-1& -2& -2& -2& -2& -2& -2\\
35& 0& 0& 0& 0& 0& 0\\
71& 4& 4& 0& 0& -2&-2\\
107& 4& 4& 0& 0& -2& -2\\
143& 6& 6& -2& -2& 0& 0\\
179& 12& 12& 0&0& 0& 0\\
215& 16& 16& 0& 0& -2& -2\\
251& 20& 20& 0& 0& 2& 2\\
287& 30&30& -2& -2& 0& 0\\
323& 40& 40& 0& 0& -2& -2\\
359& 54& 54& 2& 2& 0& 0\\
395& 72& 72& 0& 0& 0& 0\\
431& 92& 92& 0& 0& -4& -4\\
467& 116& 116& 0& 0&2& 2\\
503& 156& 156& 0& 0& 0& 0\\
539& 196& 196& 0& 0& -2& -2\\
575& 242&242& -2& -2& 2& 2\\
611& 312& 312& 0& 0& 0& 0\\
647& 388& 388& 0& 0& -2&-2\\
683& 476& 476& 0& 0& 2& 2\\
719& 594& 594& -2& -2& 0& 0\\
755& 728&728& 0& 0& -4& -4\\
791& 890& 890& 2& 2& 2& 2\\
827& 1092& 1092& 0& 0& 0&0\\
863& 1322& 1322& -2& -2& -4& -4\\
899& 1600& 1600& 0& 0& 4& 4\\
935&1938& 1938& 2& 2& 0& 0\\
971& 2324& 2324& 0& 0& -4& -4\\
1007& 2782& 2782&-2& -2& 4& 4\\
1043& 3336& 3336& 0& 0& 0& 0\\
1079& 3978& 3978& 2& 2& -6&-6\\
1115& 4720& 4720& 0& 0& 4& 4\\
1151& 5610& 5610& -2& -2& 0& 0\\
1187&6636& 6636& 0& 0& -6& -6\\
1223& 7830& 7830& 2& 2& 6& 6\\
1259& 9236& 9236&0& 0& 2& 2\\
1295& 10846& 10846& -2& -2& -8& -8\\
1331& 12720& 12720& 0& 0&6& 6\\
\bottomrule
\end{tabular}
\end{minipage}
\begin{minipage}[b]{0.49\linewidth}
\centering
\caption{\label{tab:coeffs:9_2}$H^{(9)}_{g,2}$, $\rs=A_8^3$}\smallskip
\begin{tabular}{r|rrrrrr}
\toprule
$[g]$	&1A	&2A	&2B	&2C	&3A	&6A	\\
\midrule
	$\G_g$&		$1$&$1|4$&$2$&$2|2$&	${3|3}$&	${3|12}$\\	
\midrule
32& 4& -4& 0& 0& -2& 2\\     
68& 6& -6& 2& -2& 0& 0\\     
104& 8& -8& 0& 0& 2& -2\\     
140& 16& -16& 0& 0& -2& 2\\     
176& 18& -18& -2& 2& 0& 0\\     
212& 30& -30&2& -2& 0& 0\\     
248& 40& -40& 0& 0& -2& 2\\    
284& 54& -54& 2& -2& 0& 0\\     
320& 74& -74& -2& 2& 2& -2\\     
356& 102& -102& 2& -2& 0& 0\\     
392& 126& -126& -2& 2& 0& 0\\     
428& 174& -174& 2& -2& 0& 0\\     
464& 218& -218& -2& 2& -4& 4\\     
500& 284& -284& 4& -4& 2& -2\\     
536& 356& -356& -4& 4& 2& -2\\     
572& 460& -460& 4& -4& -2& 2\\    
 608& 564& -564& -4& 4& 0& 0\\     
644& 716& -716& 4& -4& 2& -2\\     
680& 884& -884& -4& 4& -4& 4\\     
716& 1098& -1098& 6& -6& 0& 0\\     
752& 1342& -1342& -6& 6& 4& -4\\     
788& 1658& -1658& 6& -6& -4& 4\\     
824& 2000& -2000& -8& 8& 2& -2\\     
860& 2456& -2456& 8& -8& 2& -2\\     
896& 2960& -2960& -8& 8& -4& 4\\     
932& 3582& -3582& 10& -10& 0& 0\\    
 968& 4294& -4294& -10& 10& 4& -4\\     
1004& 5174& -5174& 10& -10& -4& 4\\     
1040& 6156& -6156& -12& 12& 0& 0\\     
1076& 7378& -7378& 14& -14& 4& -4\\     
1112& 8748& -8748& -12& 12& -6& 6\\     
1148& 10400& -10400& 16& -16& 2& -2\\     
1184& 12280& -12280& -16& 16& 4& -4\\     
1220& 14544& -14544& 16& -16& -6& 6\\     
1256& 17084& -17084& -20& 20& 2& -2\\     
1292& 20140& -20140& 20& -20& 4& -4\\     
1328& 23590& -23590& -22& 22& -8& 8\\     
1364& 27656& -27656& 24& -24& 2& -2\\\bottomrule
\end{tabular}
\end{minipage}
\end{table}

\begin{table}
\begin{minipage}[b]{0.49\linewidth}
\centering
\caption{\label{tab:coeffs:9_3}$H^{(9)}_{g,3}$, $\rs=A_8^3$}\smallskip
\begin{tabular}{r|@{ }r@{ }r@{ }r@{ }r@{ }r@{ }r}
\toprule
$[g]$	&1A	&2A	&2B	&2C	&3A	&6A	\\
\midrule
	$\G_g$&		$1$&$1|4$&$2$&$2|2$&	${3|3}$&	${3|12}$\\	
\midrule
27& 4& 4& 0& 0& -2& -2\\
63& 6& 6& -2& -2& 0& 0\\99& 12& 12& 0& 0& 0&0\\135& 18& 18& 2& 2& 0& 0\\171& 24& 24& 0& 0& 0& 0\\207& 36& 36& 0&0& 0& 0\\
243& 52& 52& 0& 0& -2& -2\\
279& 72& 72& 0& 0& 0& 0\\315& 96&96& 0& 0& 0& 0\\351& 126& 126& -2& -2& 0& 0\\387& 168& 168& 0& 0& 0& 0\\
423& 222& 222& 2& 2& 0& 0\\459& 288& 288& 0& 0& 0& 0\\495& 366&366& -2& -2& 0& 0\\531& 468& 468& 0& 0& 0& 0\\
567& 594& 594& 2& 2& 0& 0\\603& 744& 744& 0& 0& 0& 0\\639& 930& 930& -2& -2& 0& 0\\
675& 1156&1156& 0& 0& -2& -2\\
711& 1434& 1434& 2& 2& 0& 0\\747& 1764& 1764& 0& 0& 0& 0\\783& 2160& 2160& 0& 0& 0& 0\\819& 2640& 2640& 0& 0& 0& 0\\
855& 3210& 3210& 2& 2& 0& 0\\891& 3888& 3888& 0& 0& 0& 0\\
927& 4692& 4692& -4& -4& 0& 0\\963& 5652& 5652& 0& 0& 0& 0\\
999& 6786& 6786& 2& 2&0& 0\\1035& 8112& 8112& 0& 0& 0& 0\\1071& 9678& 9678& -2& -2& 0& 0\\
1107& 11520& 11520& 0& 0& 0& 0\\1143& 13674& 13674& 2& 2& 0& 0\\
1179& 16188& 16188& 0& 0& 0& 0\\1215& 19116& 19116& -4& -4& 0&  0\\
1251& 22536& 22536& 0& 0& 0& 0\\1287& 26508& 26508& 4& 4& 0& 0\\
1323& 31108& 31108& 0& 0& -2& -2\\
1359& 36438& 36438& -2& -2& 0& 0\\
\bottomrule
\end{tabular}
\end{minipage}
\begin{minipage}[b]{0.49\linewidth}
\centering
\caption{\label{tab:coeffs:9_4}$H^{(9)}_{g,4}$, $\rs=A_8^3$}\smallskip
\begin{tabular}{r|@{ }r@{ }r@{ }r@{ }r@{ }r@{ }r}
\toprule
$[g]$	&1A	&2A	&2B	&2C	&3A	&6A	\\
\midrule
	$\G_g$&		$1$&$1|4$&$2$&$2|2$&	${3|3}$&	${3|12}$\\	
\midrule
20& 2& -2& -2& 2& 2& -2\\56& 8& -8& 0& 0& 2& -2\\92& 10& -10& -2& 2& -2&2\\128& 18& -18& 2& -2& 0& 0\\164& 26& -26& -2& 2& 2& -2\\200& 42& -42& 2& -2& 0& 0\\236& 50& -50& -2& 2& 2& -2
\\272& 78& -78& 2& -2& 0& 0\\308& 100& -100& -4& 4& -2& 2\\344& 140& -140& 4& -4& 2& -2\\380&180& -180& -4& 4& 0& 0\\416& 244& -244& 4& -4& -2& 2\\452&302& -302& -6& 6& 2& -2\\488& 404& -404& 4& -4& 2& -2\\524&502& -502& -6& 6& -2& 2\\560& 648& -648& 8& -8& 0& 0\\
596&806& -806& -6& 6& 2& -2\\632& 1024& -1024& 8& -8& -2& 2\\668&1250& -1250& -10& 10& 2& -2\\704& 1574& -1574& 10& -10& 2& -2\\
740& 1916& -1916& -12& 12& -4& 4\\
776& 2372& -2372& 12& -12& 2& -2\\812& 2876& -2876& -12& 12& 2& -2\\848& 3530& -3530& 14& -14& -4& 4\\
884& 4240& -4240& -16& 16& 4& -4\\920& 5168& -5168& 16& -16& 2& -2\\
956& 6186& -6186& -18& 18& -6& 6\\992& 7460& -7460& 20& -20& 2& -2\\
1028& 8894& -8894& -22& 22& 2& -2\\1064& 10664& -10664& 24& -24& -4& 4\\1100& 12634& -12634& -26& 26& 4& -4\\
1136& 15070& -15070& 26& -26& 4& -4\\
1172& 17790& -17790& -30& 30& -6& 6\\1208& 21080& -21080& 32& -32& 2& -2\\1244& 24794& -24794& -34& 34& 2& -2\\
1280& 29250& -29250& 38& -38& -6& 6\\
1316& 34248& -34248& -40& 40&6& -6\\1352& 40234& -40234& 42& -42& 4& -4\\
\bottomrule
\end{tabular}
\end{minipage}
\end{table}

\begin{table}
\begin{minipage}[b]{0.49\linewidth}
\centering
\caption{\label{tab:coeffs:9_5}$H^{(9)}_{g,5}$, $\rs=A_8^3$}\smallskip
\begin{tabular}{r|@{ }r@{ }r@{ }r@{ }r@{ }r@{ }r}
\toprule
$[g]$	&1A	&2A	&2B	&2C	&3A	&6A	\\
\midrule
	$\G_g$&		$1$&$1|4$&$2$&$2|2$&	${3|3}$&	${3|12}$\\	
\midrule
11&4&4&0&0&-2&-2\\ 47&6&6&2&2&0&0\\ 83&12&12&0&0&0&0\\ 119&16&16&0&0&-2&-2\\ 155&24&24&0&0&0&0\\ 191&36&36&0&0&0&0\\ 227&52&52&0&0&-2&-2\\ 263&68&68&0&0&2&2\\ 299&96&96&0&0&0&0\\ 335&130&130&2&2&-2&-2\\ 371&168&168&0&0&0&0\\ 407&224&224&0&0&2&2\\ 443&292&292&0&0&-2&-2\\ 479&372&372&0&0&0&0\\ 515&480&480&0&0&0&0\\ 551&608&608&0&0&-4&-4\\ 587&764&764&0&0&2&2\\ 623&962&962&2&2&2&2\\ 659&1196&1196&0&0&-4&-4\\ 695&1478&1478&-2&-2&2&2\\ 731&1832&1832&0&0&2&2\\ 767&2248&2248&0&0&-2&-2\\ 803&2744&2744&0&0&2&2\\ 839&3348&3348&0&0&0&0\\ 875&4064&4064&0&0&-4&-4\\ 911&4910&4910&2&2&2&2\\ 947&5924&5924&0&0&2&2\\ 983&7116&7116&0&0&-6&-6\\ 1019&8516&8516&0&0&2&2\\ 1055&10184&10184&0&0&2&2\\ 1091&12132&12132&0&0&-6&-6\\ 1127&14404&14404&0&0&4&4\\ 1163&17084&17084&0&0&2&2\\ 1199&20202&20202&2&2&-6&-6\\ 1235&23824&23824&0&0&4&4\\ 1271&28054&28054&-2&-2&4&4\\ 1307&32956&32956&0&0&-8&-8\\ 1343&38626&38626&2&2&4&4\\
\bottomrule
\end{tabular}
\end{minipage}
\begin{minipage}[b]{0.49\linewidth}
\centering
\caption{\label{tab:coeffs:9_6}$H^{(9)}_{g,6}$, $\rs=A_8^3$}\smallskip
\begin{tabular}{r|r@{ }r@{ }r@{ }r@{ }r@{ }r}
\toprule
$[g]$	&1A	&2A	&2B	&2C	&3A	&6A	\\
\midrule
	$\G_g$&		$1$&$1|4$&$2$&$2|2$&	${3|3}$&	${3|12}$\\	
\midrule
  0&1&-1&-1&1&1&-1\\ 
  36&6&-6&2&-2&0&0\\ 72&6&-6&-2&2&0&0\\ 
  108&14&-14&2&-2&2&-2\\ 
  144&18&-18&-2&2&0&0\\ 180&30&-30&2&-2&0&0\\ 
  216&36&-36&-4&4&0&0\\ 252&60&-60&4&-4&0&0\\ 288&72&-72&-4&4&0&0\\ 324&108&-108&4&-4&0&0\\ 
  360&132&-132&-4&4&0&0\\396&186&-186&6&-6&0&0\\ 
  432&230&-230&-6&6&2&-2\\ 
  468&312&-312&8&-8&0&0\\ 504&384&-384&-8&8&0&0\\ 540&504&-504&8&-8&0&0\\ 576&618&-618&-10&10&0&0\\ 612&798&-798&10&-10&0&0\\ 648&972&-972&-12&12&0&0\\ 684&1236&-1236&12&-12&0&0\\ 720&1494&-1494&-14&14&0&0\\ 756&1872&-1872&16&-16&0&0\\ 792&2256&-2256&-16&16&0&0\\ 828&2790&-2790&18&-18&0&0\\ 864&3348&-3348&-20&20&0&0\\ 900&4098&-4098&22&-22&0&0\\ 936&4896&-4896&-24&24&0&0\\ 
  972&5942&-5942&26&-26&2&-2\\ 
  1008&7068&-7068&-28&28&0&0\\ 1044&8514&-8514&30&-30&0&0\\ 1080&10080&-10080&-32&32&0&0\\ 1116&12060&-12060&36&-36&0&0\\ 1152&14226&-14226&-38&38&0&0\\ 1188&16920&-16920&40&-40&0&0\\ 1224&19884&-19884&-44&44&0&0\\ 1260&23520&-23520&48&-48&0&0\\ 1296&27540&-27540&-52&52&0&0\\ 1332&32424&-32424&56&-56&0&0\\
\bottomrule
\end{tabular}
\end{minipage}
\end{table}

\begin{table}
\begin{minipage}[b]{0.49\linewidth}
\centering
\caption{\label{tab:coeffs:9_7}$H^{(9)}_{g,7}$, $\rs=A_8^3$}\smallskip
\begin{tabular}{r|@{ }r@{ }r@{ }r@{ }r@{ }r@{ }r}
\toprule
$[g]$	&1A	&2A	&2B	&2C	&3A	&6A	\\
\midrule
	$\G_g$&		$1$&$1|4$&$2$&$2|2$&	${3|3}$&	${3|12}$\\	
\midrule
  23&2&2&-2&-2&2&2\\ 59&4&4&0&0&-2&-2\\ 95&8&8&0&0&2&2\\ 131&12&12&0&0&0&0\\ 167&18&18&-2&-2&0&0\\ 203&24&24&0&0&0&0\\ 239&38&38&2&2&2&2\\ 275&52&52&0&0&-2&-2\\ 311&66&66&-2&-2&0&0\\ 347&92&92&0&0&2&2\\ 383&124&124&0&0&-2&-2\\ 419&156&156&0&0&0&0
  \\ 455&206&206&-2&-2&2&2\\ 491&268&268&0&0&-2&-2\\ 527&338&338&2&2&2&2\\ 563&428&428&0&0&2&2\\ 599&538&538&-2&-2&-2&-2\\ 635&672&672&0&0&0&0\\ 671&842&842&2&2&2&2\\ 707&1040&1040&0&0&-4&-4\\ 743&1274&1274&-2&-2&2&2\\ 779&1568&1568&0&0&2&2\\ 815&1922&1922&2&2&-4&-4\\ 851&2328&2328&0&0&0&0\\ 887&2824&2824&-4&-4&4&4\\ 923&3416&3416&0&0&-4&-4\\ 959&4106&4106&2&2&2&2\\ 995&4936&4936&0&0&4&4\\ 1031&5906&5906&-2&-2&-4&-4\\ 1067&7040&7040&0&0&2&2\\ 1103&8392&8392&4&4&4&4\\ 1139&9960&9960&0&0&-6&-6
  \\ 1175&11784&11784&-4&-4&0&0\\ 1211&13936&13936&0&0&4&4\\ 1247&16434&16434&2&2&-6&-6\\ 1283&19316&19316&0&0&2&2\\ 1319&22680&22680&-4&-4&6&6\\	
\bottomrule
\end{tabular}
\end{minipage}
\begin{minipage}[b]{0.49\linewidth}
\centering
\caption{\label{tab:coeffs:9_8}$H^{(9)}_{g,8}$, $\rs=A_8^3$}\smallskip
\begin{tabular}{r|@{ }r@{ }r@{ }r@{ }r@{ }r@{ }r}
\toprule
$[g]$	&1A	&2A	&2B	&2C	&3A	&6A	\\
\midrule
	$\G_g$&		$1$&$1|4$&$2$&$2|2$&	${3|3}$&	${3|12}$\\	
\midrule
8&2&-2&2&-2&2&-2\\ 44&2&-2&-2&2&2&-2\\ 80&6&-6&2&-2&0&0\\ 116&2&-2&-2&2&2&-2\\ 152&12&-12&4&-4&0&0\\ 188&10&-10&-2&2&-2&2\\ 224&20&-20&4&-4&2&-2\\ 260&20&-20&-4&4&2&-2\\ 296&36&-36&4&-4&0&0\\ 332&38&-38&-6&6&2&-2\\ 368&66&-66&6&-6&0&0\\ 404&70&-70&-6&6&-2&2\\ 440&104&-104&8&-8&2&-2\\ 476&120&-120&-8&8&0&0\\ 512&172&-172&8&-8&-2&2\\ 548&194&-194&-10&10&2&-2\\ 584&276&-276&12&-12&0&0\\ 620&316&-316&-12&12&-2&2\\ 656&418&-418&14&-14&4&-4\\ 692&494&-494&-14&14&2&-2\\ 728&640&-640&16&-16&-2&2\\ 764&746&-746&-18&18&2&-2\\ 800&960&-960&20&-20&0&0\\ 836&1124&-1124&-20&20&-4&4\\ 872&1408&-1408&24&-24&4&-4\\ 908&1658&-1658&-26&26&2&-2\\ 944&2054&-2054&26&-26&-4&4\\ 980&2398&-2398&-30&30&4&-4\\ 1016&2952&-2952&32&-32&0&0\\ 1052&3450&-3450&-34&34&-6&6\\ 1088&4186&-4186&38&-38&4&-4\\ 1124&4898&-4898&-42&42&2&-2\\ 1160&5892&-5892&44&-44&-6&6\\ 1196&6864&-6864&-48&48&6&-6\\ 1232&8220&-8220&52&-52&0&0\\ 1268&9558&-9558&-54&54&-6&6\\ 1304&11348&-11348&60&-60&8&-8\\
\bottomrule
\end{tabular}
\end{minipage}
\end{table}

\clearpage

\begin{table}
\section{Decompositions at $\ell=9$}\label{sec:decompositions}
\begin{minipage}[t]{0.5\linewidth}
 \begin{center}
     \caption{  $K^{(9)}_1$, $X=A_8^{3}  $}\label{tab:dec:9_1}
    \vspace{0.1cm}
 \begin{tabular}{cc|RRRR}
  \toprule  && \chi_{1}& \chi_{2}& \chi_{3} \\\midrule
& -1& -2& 0& 0\\
 &35& 0& 0& 0\\ 
&71& 0& 0& 2\\ 
&107& 0& 0& 2\\ 
&143& 0& 2& 2\\ 
&179& 2& 2& 4\\ 
&215& 2& 2& 6\\ 
&251& 4& 4& 6\\ 
&287& 4& 6& 10\\ 
&323& 6& 6& 14
\\\bottomrule
\end{tabular}\end{center}
 \begin{center}
     \caption{  $K^{(9)}_5$, $X=A_8^{3}  $}\label{tab:dec:9_5}
    \vspace{0.1cm}

 \begin{tabular}{c|RRRR}
 \toprule  & \chi_{1}& \chi_{2}& \chi_{3}\\\midrule
11& 0& 0&2\\ 
 47& 2& 0& 2\\
 83& 2& 2& 4\\
 119& 2& 2& 6\\
 155& 4& 4& 8\\
191& 6& 6& 12\\ 
227& 8& 8& 18\\
 263& 12& 12& 22\\
 299& 16& 16& 32\\
 335& 22& 20& 44
 \\\bottomrule
\end{tabular}\end{center}
 \begin{center}
     \caption{  $K^{(9)}_7$, $X=A_8^{3}  $}\label{tab:dec:9_7}
    \vspace{0.1cm}

 \begin{tabular}{c|RRRR}
 \toprule  & \chi_{1}& \chi_{2}& \chi_{3} \\\midrule
23& 0& 2& 0\\ 
59& 0& 0& 2\\
95& 2& 2& 2\\
131& 2& 2& 4\\
167& 2& 4& 6\\
203& 4& 4& 8\\
239& 8& 6& 12\\
275& 8& 8& 18\\
311& 10& 12& 22\\
347& 16& 16& 30\\
 \bottomrule
\end{tabular}\end{center}
\end{minipage}
\begin{minipage}[t]{0.5\linewidth}
 \begin{center}
     \caption{  $K^{(9)}_2$, $X=A_8^{3}  $}\label{tab:dec:9_2}
    \vspace{0.1cm}

 \begin{tabular}{c|RRR}
 \toprule  & \chi_{5}& \chi_{6}& \chi_{7} \\\midrule
32&0&0&2\\68&0&2&2\\104&2&2&2\\140&2&2&6\\176&4&2&6\\212&4&6&10\\248&6&6&14
\\284&8&10&18\\320&14&12&24\\356&16&18&34
  \\\bottomrule
\end{tabular}\end{center}
 \begin{center}
     \caption{  $K^{(9)}_4$, $X=A_8^{3}  $}\label{tab:dec:9_4}
    \vspace{0.1cm}

  \begin{tabular}{c|RRR}
   \toprule  & \chi_{5}& \chi_{6}& \chi_{7} \\\midrule
 20&2&0&0\\56&2&2&2\\92&2&0&4\\128&2&4&6\\164&6&4&8\\200&6&8&14\\236&10&8&16\\272&12&14&26\\308&18&14&34\\344&22&26&46 \\\bottomrule
\end{tabular}\end{center}
 \begin{center}
     \caption{  $K^{(9)}_8$, $X=A_8^{3}  $}\label{tab:dec:9_8}
    \vspace{0.1cm}

 \begin{tabular}{c|RRR}
 \toprule  & \chi_{5}& \chi_{6}& \chi_{7} \\\midrule
8&0&2&0\\44&2&0&0\\80&0&2&2\\116&2&0&0\\152&0&4&4\\188&2&0&4\\224&2&6&6\\260&6&2&6\\296&4&8&12\\332&10&4&12\\
  \bottomrule \end{tabular}\end{center}
\end{minipage}\end{table}

\begin{table} \begin{minipage}[t]{0.5\linewidth}
 \begin{center}
     \caption{  $K^{(9)}_3$, $X=A_8^{3}  $}\label{tab:dec:9_3}
    \vspace{0.1cm}

 \begin{tabular}{c|RRRR}
 \toprule  & \chi_{1}& \chi_{2}& \chi_{3} \\\midrule
27&0&0&2\\63&0&2&2\\99&2&2&4\\135&4&2&6\\171&4&4&8\\207&6&6&12\\243&8&8&18\\279&12&12&24\\315&16&16&32\\351&20&22&42\\387&28&28&56\\423&38&36&74\\459&48&48&96\\495&60&62&122\\531&78&78&156\\567&100&98&198\\603&124&124&248\\639&154&156&310\\675&192&192&386       \\\bottomrule
\end{tabular}\end{center}
\end{minipage}
\begin{minipage}[t]{0.5\linewidth}
 \begin{center}
     \caption{  $K^{(9)}_6$, $X=A_8^{3}  $}\label{tab:dec:9_6}
    \vspace{0.1cm}

 \begin{tabular}{c|RRR}
  \toprule  & \chi_{4}& \chi_{5}& \chi_{6} \\\midrule
 0&1&0&0\\36&0&2&2\\72&2&0&2\\108&2&4&4\\144&4&2&6\\180&4&6&10\\216&8&4&12\\252&8&12&20\\288&14&10&24\\324&16&20&36\\360&24&20&44\\396&28&34&62\\432&42&36&76\\468&48&56&104\\504&68&60&128\\540&80&88&168\\576&108&98&206\\612&128&138&266\\648&168&156&324\\\bottomrule
\end{tabular}\end{center}
\end{minipage}\end{table}

\clearpage

\providecommand{\href}[2]{#2}\begingroup\raggedright\endgroup


\begin{thebibliography}{10}

\bibitem{UM}
M.~C.~N. Cheng, J.~F.~R. Duncan, and J.~A. Harvey, ``{Umbral Moonshine},''
  \href{http://dx.doi.org/10.4310/CNTP.2014.v8.n2.a1}{{\em Commun. Number
  Theory Phys.} {\bf 8} (2014) no.~2, 101--242},
  \href{http://arxiv.org/abs/1204.2779}{{\tt arXiv:1204.2779 [math.RT]}}.
\url{http://dx.doi.org/10.4310/CNTP.2014.v8.n2.a1}.

\bibitem{MUM}
M.~C.~N. Cheng, J.~F.~R. Duncan, and J.~A. Harvey, ``{Umbral Moonshine and the
  Niemeier Lattices},'' {\em Research in the Mathematical Sciences} {\bf 1}
  (2014) no.~3, 1--81,
\href{http://arxiv.org/abs/1307.5793}{{\tt arXiv:1307.5793 [math.RT]}}.

\bibitem{Nie_DefQdtFrm24}
H.-V. Niemeier, ``Definite quadratische {F}ormen der {D}imension {$24$} und
  {D}iskriminante {$1$},'' {\em J. Number Theory} {\bf 5} (1973)  142--178.

\bibitem{MR1662447}
J.~H. Conway and N.~J.~A. Sloane,
  \href{http://dx.doi.org/10.1007/978-1-4757-6568-7}{{\em Sphere packings,
  lattices and groups}}, vol.~290 of {\em Grundlehren der Mathematischen
  Wissenschaften [Fundamental Principles of Mathematical Sciences]}.
\newblock Springer-Verlag, New York, third~ed., 1999.
\newblock \url{http://dx.doi.org/10.1007/978-1-4757-6568-7}.
\newblock With additional contributions by E. Bannai, R. E. Borcherds, J.
  Leech, S. P. Norton, A. M. Odlyzko, R. A. Parker, L. Queen and B. B. Venkov.

\bibitem{umrec}
J.~F.~R. {Duncan}, M.~J. {Griffin}, and K.~{Ono}, ``{Proof of the Umbral
  Moonshine Conjecture},''{\em Research in the Mathematical Sciences} {\bf 2}
  (Mar., 2015)  , \href{http://arxiv.org/abs/1503.01472}{{\tt arXiv:1503.01472
  [math.RT]}}.

\bibitem{MR554399}
J.~H. Conway and S.~P. Norton, ``Monstrous moonshine,''
  \href{http://dx.doi.org/10.1112/blms/11.3.308}{{\em Bull. London Math. Soc.}
  {\bf 11} (1979) no.~3, 308--339}.
  \url{http://dx.doi.org/10.1112/blms/11.3.308}.

\bibitem{borcherds_monstrous}
R.~E. Borcherds, ``{Monstrous moonshine and monstrous Lie superalgebras},''
  {\em Invent. Math.} {\bf 109, No.2} (1992)  405--444.

\bibitem{FLMPNAS}
I.~B. Frenkel, J.~Lepowsky, and A.~Meurman, ``A natural representation of the
  {F}ischer-{G}riess {M}onster with the modular function {$J$} as character,''
  {\em Proc. Nat. Acad. Sci. U.S.A.} {\bf 81} (1984) no.~10, Phys. Sci.,
  3256--3260.

\bibitem{FLMBerk}
I.~B. Frenkel, J.~Lepowsky, and A.~Meurman, ``A moonshine module for the
  {M}onster,'' in {\em Vertex operators in mathematics and physics (Berkeley,
  Calif., 1983)}, vol.~3 of {\em Math. Sci. Res. Inst. Publ.}, pp.~231--273.
\newblock Springer, New York, 1985.

\bibitem{FLM}
I.~B. Frenkel, J.~Lepowsky, and A.~Meurman, {\em Vertex operator algebras and
  the {M}onster}, vol.~134 of {\em Pure and Applied Mathematics}.
\newblock Academic Press Inc., Boston, MA, 1988.

\bibitem{Dabholkar:2012nd}
A.~Dabholkar, S.~Murthy, and D.~Zagier, ``{Quantum Black Holes, Wall Crossing,
  and Mock Modular Forms},''
\href{http://arxiv.org/abs/1208.4074}{{\tt arXiv:1208.4074 [hep-th]}}.

\bibitem{Cheng2011}
M.~C.~N. Cheng and J.~F.~R. Duncan, ``{On Rademacher Sums, the Largest Mathieu
  Group, and the Holographic Modularity of Moonshine},''
  \href{http://dx.doi.org/10.4310/CNTP.2012.v6.n3.a4}{{\em Commun. Number
  Theory Phys.} {\bf 6} (2012) no.~3, 697--758},
  \href{http://arxiv.org/abs/1110.3859}{{\tt 1110.3859}}.
  \url{http://dx.doi.org/10.4310/CNTP.2012.v6.n3.a4}.

\bibitem{Sko_Thesis}
N.-P. Skoruppa, {\em {\"U}ber den {Z}usammenhang zwischen {J}acobiformen und
  {M}odulformen halbganzen {G}ewichts}.
\newblock Bonner Mathematische Schriften [Bonn Mathematical Publications], 159.
  Universit{\"a}t Bonn Mathematisches Institut, Bonn, 1985.
\newblock Dissertation, Rheinische Friedrich-Wilhelms-Universit{\"a}t, Bonn,
  1984.

\bibitem{MR2545599}
R.~Schmidt, ``A remark on a paper of {I}bukiyama and {S}koruppa,''
  \href{http://dx.doi.org/10.1007/s12188-009-0026-z}{{\em Abh. Math. Semin.
  Univ. Hambg.} {\bf 79} (2009) no.~2, 189--191}.
  \url{http://dx.doi.org/10.1007/s12188-009-0026-z}.

\bibitem{GriNik_AutFrmLorKMAlgs_II}
V.~A. Gritsenko and V.~V. Nikulin, ``Automorphic forms and {L}orentzian
  {K}ac-{M}oody algebras. {II},''
  \href{http://dx.doi.org/10.1142/S0129167X98000117}{{\em Internat. J. Math.}
  {\bf 9} (1998) no.~2, 201--275}.
  \url{http://dx.doi.org/10.1142/S0129167X98000117}.

\bibitem{MR2806099}
F.~Cl\'ery and V.~Gritsenko, ``Siegel modular forms of genus 2 with the
  simplest divisor,'' \href{http://dx.doi.org/10.1112/plms/pdq036}{{\em Proc.
  Lond. Math. Soc. (3)} {\bf 102} (2011) no.~6, 1024--1052}.
  \url{http://dx.doi.org/10.1112/plms/pdq036}.

\bibitem{MR3123592}
F.~Cl{\'e}ry and V.~Gritsenko, ``Modular forms of orthogonal type and {J}acobi
  theta-series,'' \href{http://dx.doi.org/10.1007/s12188-013-0080-4}{{\em Abh.
  Math. Semin. Univ. Hambg.} {\bf 83} (2013) no.~2, 187--217}.
  \url{http://dx.doi.org/10.1007/s12188-013-0080-4}.

\bibitem{GSZ_ThtBlk}
V.~{Gritsenko}, N.-P. {Skoruppa}, and D.~{Zagier}, ``Theta blocks,'' {\em
  (preprint)}  .

\bibitem{MR2379341}
T.~Ibukiyama and N.-P. Skoruppa, ``{A Vanishing Theorem for Siegel Modular
  Forms of Weight One},'' \href{http://dx.doi.org/10.1007/BF03173501}{{\em Abh.
  Math. Sem. Univ. Hamburg} {\bf 77} (2007)  229--235}.
  \url{http://dx.doi.org/10.1007/BF03173501}.

\bibitem{IbuSko_Cor}
T.~Ibukiyama and N.-P. Skoruppa, ``Correction to ``{A} {V}anishing {T}heorem
  for {S}iegel {M}odular {F}orms of {W}eight {O}ne'','' {\em (preprint)}  .

\bibitem{mnstmlts}
J.~F.~R. {Duncan}, M.~J. {Griffin}, and K.~{Ono}, ``{Moonshine},'' {\em
  Research in the Mathematical Sciences} {\bf 2} (2015) no.~11, ,
  \href{http://arxiv.org/abs/1411.6571}{{\tt arXiv:1411.6571 [math.RT]}}.

\bibitem{MR2512363}
N.-P. Skoruppa, \href{http://dx.doi.org/10.1017/CBO9780511543371.013}{``Jacobi
  forms of critical weight and {W}eil representations,''} in {\em Modular forms
  on {S}chiermonnikoog}, pp.~239--266.
\newblock Cambridge Univ. Press, Cambridge, 2008.
\newblock \url{http://dx.doi.org/10.1017/CBO9780511543371.013}.

\bibitem{MR3539377}
T.~Gannon, ``Much ado about {M}athieu,''
  \href{http://dx.doi.org/10.1016/j.aim.2016.06.014}{{\em Adv. Math.} {\bf 301}
  (2016)  322--358}. \url{http://dx.doi.org/10.1016/j.aim.2016.06.014}.

\bibitem{MR3108775}
M.~R. Gaberdiel, D.~Persson, H.~Ronellenfitsch, and R.~Volpato, ``Generalized
  {M}athieu {M}oonshine,''
  \href{http://dx.doi.org/10.4310/CNTP.2013.v7.n1.a5}{{\em Commun. Number
  Theory Phys.} {\bf 7} (2013) no.~1, 145--223}.
  \url{http://dx.doi.org/10.4310/CNTP.2013.v7.n1.a5}.

\bibitem{Cheng:2016nto}
M.~C.~N. Cheng, P.~de~Lange, and D.~P.~Z. Whalen, ``{Generalised Umbral
  Moonshine},''
\href{http://arxiv.org/abs/1608.07835}{{\tt arXiv:1608.07835 [math.RT]}}.

\bibitem{Cheng:2014fk}
M.~C.~N. Cheng and J.~F.~R. Duncan,
  \href{http://dx.doi.org/10.1007/978-3-662-43831-2_6}{``{Rademacher Sums and
  Rademacher Series},''} in {\em Conformal Field Theory, Automorphic Forms and
  Related Topics}, W.~Kohnen and R.~Weissauer, eds., vol.~8 of {\em
  Contributions in Mathematical and Computational Sciences}, pp.~143--182.
\newblock Springer Berlin Heidelberg, 2014.
\newblock \url{http://dx.doi.org/10.1007/978-3-662-43831-2_6}.

\bibitem{2014arXiv1406.0571W}
D.~{Whalen}, ``{Vector-Valued Rademacher Sums and Automorphic Integrals},''{\em
  ArXiv e-prints} (June, 2014)  , \href{http://arxiv.org/abs/1406.0571}{{\tt
  arXiv:1406.0571 [math.NT]}}.

\bibitem{MR1277050}
V.~Gritsenko, ``Irrationality of the moduli spaces of polarized abelian
  surfaces,'' \href{http://dx.doi.org/10.1155/S1073792894000267}{{\em Internat.
  Math. Res. Notices} (1994) no.~6, 235 ff., approx.\ 9 pp.\ (electronic)}.
  \url{http://dx.doi.org/10.1155/S1073792894000267}.

\bibitem{MR0469872}
E.~Freitag, ``Siegelsche {M}odulfunktionen,'' {\em Jber. Deutsch.
  Math.-Verein.} {\bf 79} (1977) no.~3, 79--86.

\bibitem{feingold_frenkel}
A.~J. Feingold and I.~B. Frenkel, ``{A hyperbolic Kac-Moody algebra and the
  theory of Siegel modular forms of genus 2},'' {\em J. Math. Ann.} {\bf 263}
  (1983)  87--144.

\bibitem{MR582704}
H.~Maa{\ss}, ``\"{U}ber ein {A}nalogon zur {V}ermutung von
  {S}aito-{K}urokawa,'' \href{http://dx.doi.org/10.1007/BF01389899}{{\em
  Invent. Math.} {\bf 60} (1980) no.~1, 85--104}.
  \url{http://dx.doi.org/10.1007/BF01389899}.

\bibitem{MR1345176}
V.~Gritsenko,
  \href{http://dx.doi.org/10.1017/CBO9780511661990.008}{``Arithmetical lifting
  and its applications,''} in {\em Number theory ({P}aris, 1992--1993)},
  vol.~215 of {\em London Math. Soc. Lecture Note Ser.}, pp.~103--126.
\newblock Cambridge Univ. Press, Cambridge, 1995.
\newblock \url{http://dx.doi.org/10.1017/CBO9780511661990.008}.

\bibitem{MR1030139}
K.~G. O'Grady, ``On the {K}odaira dimension of moduli spaces of abelian
  surfaces,'' {\em Compositio Math.} {\bf 72} (1989) no.~2, 121--163.
  \url{http://www.numdam.org/item?id=CM_1989__72_2_121_0}.

\bibitem{MR1827859}
M.~Gross and S.~Popescu, ``The moduli space of {$(1,11)$}-polarized abelian
  surfaces is unirational,''
  \href{http://dx.doi.org/10.1023/A:1017518027822}{{\em Compositio Math.} {\bf
  126} (2001) no.~1, 1--23}. \url{http://dx.doi.org/10.1023/A:1017518027822}.

\bibitem{MR3031888}
C.~Poor and D.~S. Yuen, ``The cusp structure of the paramodular groups for
  degree two,'' \href{http://dx.doi.org/10.4134/JKMS.2013.50.2.445}{{\em J.
  Korean Math. Soc.} {\bf 50} (2013) no.~2, 445--464}.
  \url{http://dx.doi.org/10.4134/JKMS.2013.50.2.445}.

\bibitem{MR3283174}
V.~Gritsenko, C.~Poor, and D.~S. Yuen, ``Borcherds products everywhere,''
  \href{http://dx.doi.org/10.1016/j.jnt.2014.07.028}{{\em J. Number Theory}
  {\bf 148} (2015)  164--195}.
  \url{http://dx.doi.org/10.1016/j.jnt.2014.07.028}.

\bibitem{MR3498287}
J.~Breeding, II, C.~Poor, and D.~S. Yuen, ``Computations of spaces of
  paramodular forms of general level,''
  \href{http://dx.doi.org/10.4134/JKMS.j150219}{{\em J. Korean Math. Soc.} {\bf
  53} (2016) no.~3, 645--689}. \url{http://dx.doi.org/10.4134/JKMS.j150219}.

\bibitem{LopesCardoso:1996zj}
G.~Lopes~Cardoso, ``{Perturbative gravitational couplings and Siegel modular
  forms in D = 4, N=2 string models},'' {\em Nucl.Phys.Proc.Suppl.} {\bf 56B}
  (1997)  94--101,
\href{http://arxiv.org/abs/hep-th/9612200}{{\tt arXiv:hep-th/9612200
  [hep-th]}}.

\bibitem{MR1668093}
C.~D.~D. Neumann, ``The elliptic genus of {C}alabi-{Y}au {$3$}- and
  {$4$}-folds, product formulae and generalized {K}ac-{M}oody algebras,''
  \href{http://dx.doi.org/10.1016/S0393-0440(98)00015-1}{{\em J. Geom. Phys.}
  {\bf 29} (1999) no.~1-2, 5--12}.
  \url{http://dx.doi.org/10.1016/S0393-0440(98)00015-1}.

\bibitem{MR3158917}
C.~Nazaroglu, ``Jacobi forms of higher index and paramodular groups in
  {$\mathcal{N}=2$}, {$D=4$} compactifications of string theory,'' {\em J. High
  Energy Phys.} (2013) no.~12, 074, front matter + 49.

\bibitem{omjt}
M.~C.~N. {Cheng} and J.~F.~R. {Duncan}, ``{Optimal Mock Jacobi Theta
  Functions},''{\em ArXiv e-prints} (May, 2016)  ,
  \href{http://arxiv.org/abs/1605.04480}{{\tt arXiv:1605.04480 [math.NT]}}.

\bibitem{MR3582425}
M.~J. Griffin and M.~H. Mertens, ``A proof of the {T}hompson moonshine
  conjecture,'' \href{http://dx.doi.org/10.1186/s40687-016-0084-7}{{\em Res.
  Math. Sci.} {\bf 3} (2016)  3:36}.
  \url{http://dx.doi.org/10.1186/s40687-016-0084-7}.

\bibitem{BruFun_TwoGmtThtLfts}
J.~H. Bruinier and J.~Funke, ``On two geometric theta lifts,''
  \href{http://dx.doi.org/10.1215/S0012-7094-04-12513-8}{{\em Duke Math. J.}
  {\bf 125} (2004) no.~1, 45--90}.
  \url{http://dx.doi.org/10.1215/S0012-7094-04-12513-8}.

\bibitem{MR0332663}
G.~Shimura, ``On modular forms of half integral weight,'' {\em Ann. of Math.
  (2)} {\bf 97} (1973)  440--481.

\bibitem{GAP4}
The GAP~Group, {\em {GAP -- Groups, Algorithms, and Programming, Version 4.4}},
  2005.
\newblock \verb+(http://www.gap-system.org)+.

\end{thebibliography}
\end{document}